\documentclass[14pt,dvipsnames]{extarticle}

\usepackage[left=2cm,top=2cm,right=2cm,bottom=20mm]{geometry}
\usepackage[T1]{fontenc}
\usepackage[utf8]{inputenc}
\usepackage[cmintegrals,cmbraces]{newtxmath}
\usepackage[lining]{ebgaramond}
\usepackage{ebgaramond-maths}
\usepackage{euscript}
\usepackage{mathtools}
\makeatletter
\DeclareSymbolFont{ntxletters}{OML}{ntxmi}{m}{it}
\SetSymbolFont{ntxletters}{bold}{OML}{ntxmi}{b}{it}
\re@DeclareMathSymbol{\partial}{\mathord}{ntxletters}{"40}
\makeatother
\usepackage{graphicx}
\usepackage[figurename=Fig.]{caption}
\usepackage[tracking=true,letterspace=50]{microtype}
\SetTracking{encoding=T1,shape=sc}{50}

\usepackage{amsthm}
\usepackage{array,booktabs}
\usepackage{enumitem}
\usepackage{xcolor}
\usepackage{tocloft}
\usepackage{csquotes}

\usepackage[explicit]{titlesec}
\usepackage{pgfornament}

\newcommand{\sectionflourish}{%
  \pgfornament[width=3cm]{87}%
}

\titleformat{\section}[block]
  {\normalfont\Large\bfseries\filcenter} 
  {}
  {0pt}
  {\thesection.\hspace{0.45em}#1}
  [\vspace{-0.15em}\sectionflourish]

\titleformat{name=\section,numberless}[block]
  {\normalfont\Large\bfseries\filcenter}
  {}
  {0pt}
  {#1}
  [\vspace{-0.15em}\sectionflourish]

\titlespacing*{\section}
  {0pt}
  {3ex plus 0.3ex minus 0.2ex}
  {2.5ex}
  
  \usepackage{indentfirst}

\usepackage{hyperref}

\usepackage{xurl}

\usepackage[
 backend=biber,
 style=alphabetic,
 sorting=nyt,
 giveninits=true,
 maxbibnames=99,
 doi=true,
 url=true,
 isbn=false,
 eprint=true
]{biblatex}
\AtBeginBibliography{\normalsize\raggedright}
\renewbibmacro{in:}{}

\definecolor{paperred}{RGB}{125,18,18}
\hypersetup{
 colorlinks=true,
 linkcolor={red!50!black},
 citecolor={blue!50!black},
 urlcolor={blue!80!black},
 pdfauthor={Victor Buchstaber and Mikhail Kornev},
 pdftitle={Logarithm of the Universal Two-Valued Formal Group}
}

\setlist{nosep}
\allowdisplaybreaks
\numberwithin{equation}{subsection}
\renewcommand{\theequation}{%
  \ifnum\value{subsection}>0 \thesubsection\else\thesection\fi.\arabic{equation}%
}

\usepackage{aliascnt}
\usepackage[nameinlink,capitalise,noabbrev]{cleveref}
\newtheoremstyle{plain}
  {\topsep}{\topsep}{\itshape}{}{\bfseries}{.}
  {5pt plus 1pt minus 1pt}
  {\thmname{#1}\thmnumber{ \textup{#2}}\thmnote{ {\fontseries{\mddefault}\upshape #3}}}
\newtheoremstyle{definition}
  {\topsep}{\topsep}{\normalfont}{}{\bfseries}{.}
  {5pt plus 1pt minus 1pt}
  {\thmname{#1}\thmnumber{ \textup{#2}}\thmnote{ {\fontseries{\mddefault}\upshape #3}}}
\theoremstyle{plain}
\newaliascnt{theorem}{equation}
\newtheorem{theorem}[theorem]{Theorem}
\aliascntresetthe{theorem}
\newaliascnt{proposition}{equation}
\newtheorem{proposition}[proposition]{Proposition}
\aliascntresetthe{proposition}
\newaliascnt{lemma}{equation}
\newtheorem{lemma}[lemma]{Lemma}
\aliascntresetthe{lemma}
\newaliascnt{corollary}{equation}
\newtheorem{corollary}[corollary]{Corollary}
\aliascntresetthe{corollary}
\theoremstyle{definition}
\newaliascnt{definition}{equation}

\aliascntresetthe{definition}
\newaliascnt{example}{equation}
\newtheorem{example}[example]{Example}
\aliascntresetthe{example}
\newaliascnt{problem}{equation}

\aliascntresetthe{problem}
\theoremstyle{definition}
\newaliascnt{remark}{equation}
\newtheorem{remark}[remark]{Remark}
\aliascntresetthe{remark}

\makeatletter
\renewenvironment{proof}[1][\proofname]{
 \par\pushQED{\qed}\normalfont
 \topsep6\p@\@plus6\p@\relax
 \trivlist\item[\hskip\labelsep\bfseries\itshape #1\@addpunct{.}]
}{\popQED\endtrivlist\@endpefalse}
\makeatother

\let\phi\varphi

\DeclareMathOperator*{\res}{res}
\DeclareMathOperator{\lcm}{lcm}

\DeclareMathOperator{\ord}{ord}
\DeclareMathOperator{\Td}{Td}
\DeclareMathOperator{\Och}{Och}
\DeclareMathOperator{\Tors}{Tors}
\DeclareMathOperator{\pr}{pr}
\newcommand{\dd}{\mathop{}\!\mathrm d}
\newcommand{\vp}{\nu_p}
\newcommand{\OmU}{\Omega_{\mathrm U}^{\ast}}
\newcommand{\OmSp}{\Omega_{\mathrm{Sp}}^{\ast}}
\newcommand{\OmSpfree}{\widetilde{\Omega}_{\mathrm{Sp}}}

\newcommand{\U}{\mathrm U}
\newcommand{\MSp}{\mathrm{MSp}}
\newcommand{\CP}{\mathbb{C}\mathrm P}
\newcommand{\HP}{\mathbb{H}\mathrm P}
\newcommand{\Sp}{\mathrm{Sp}}

\newcommand{\Ii}{\mathcal I}
\newcommand{\Ll}{\mathcal L\mkern1mu}
\newcommand{\Uu}{\mathrm{U}}

\newcommand{\Bc}{\mathrm{Bc}}
\newcommand{\St}{\mathrm{St}}
\newcommand{\LambdaSt}{\Lambda_{\mathrm{St}}}
\newcommand{\Q}{\mathbb{Q}}

\newcommand{\Ztwo}{\mathbb{Z}_{(2)}}

\title{\bfseries Logarithm of the Universal\\ Two-Valued Formal Group}
\author{Victor Buchstaber and Mikhail Kornev}
\date{}
\begin{document}

\maketitle

\begin{abstract}
We solve the long-standing problem of determining the exact denominators of the coefficients of the logarithm $B(x)=x+\sum_{n\geq1}b_nx^{n+1}$, obtained from the universal formal group of complex cobordism by the modulus square construction:
$$
b_n=\frac{C_n}{d_n},\qquad
d_n=\frac{n+1}{2}\operatorname{lcm}(1,\ldots,2n+2),
$$
where $C_n\in\Omega_{\mathrm U}^{-4n}$ is primitive and undecomposable. We prove that $C_n$ belongs to the coefficient ring $\Lambda$ of the universal two-valued law and to the subring $\Lambda_{\mathrm{St}}$ generated by quotient Stong manifolds, and give an explicit integral Stong manifold formula for it. The rational lift of $b_n$ to quaternionic cobordism modulo torsion has exact denominator $2d_n$. Consequently, $C_n$ has no integral quaternionic lift, whereas $2C_n$ does. Every Chern number of $C_n$ is divisible by $d_n$, and $c_{2n}(C_n)=d_n$.

We introduce odd genera $\mathrm{Bc}_N$. For $N\geq n$, $\mathrm{Bc}_N$ detects the odd part of $d_n$. Its restriction to the Stong ring is integral exactly for $N\leq3$, while the universal odd genus $\mathrm{Bc}_{\infty}$ takes values there with only odd denominators. For every $n\geq5$, we construct undecomposable classes in $\Lambda^{-4n}$ with equal Ochanine genera and top Chern numbers but distinct $\mathrm{Bc}_3$ values. Finally, the oriented extension of $\mathrm{Bc}_3$ is integral on closed spin manifolds of real dimension below $24$ and on closed string manifolds of real dimension at most $24$. The spin bound is sharp: an Anderson--Brown--Peterson spin $24$-manifold has nonintegral $\mathrm{Bc}_3$ genus.
\end{abstract}

\thispagestyle{empty}

\newpage

\begin{center}
  \pgfornament[width=3cm]{87}
\end{center}

\setcounter{tocdepth}{2}
\tableofcontents
\thispagestyle{empty}
\newpage

\section{Introduction}

One of the best-known problems in complex cobordism theory concerns the
relations among the Chern characteristic numbers of stably complex manifolds.
The Milnor--Novikov theorem shows that these numbers determine a complex
cobordism class, while the Hattori--Stong theorem gives a complete system of
$\mathrm{K}$-theoretic integrality conditions distinguishing the integral complex
cobordism group from its rationalization
\cite{Mil60,Nov60,Nov62,Hat66,Sto65}. This problem is closely tied to two
classical structural questions: to determine the coefficient ring
$\Omega_{\mathrm U}^{\ast}$ and to describe the formal group law associated
with the $\mathrm{MU}$-theoretic\footnote{For a generalized cohomology theory
represented by a Thom spectrum $\mathrm{MG}$, we omit the letter $\mathrm M$
from the notation and write $\mathrm{G}^{*}(X)$ instead of
$\mathrm{MG}^{*}(X)$. For example, we write $\Uu^{*}(X)$ instead of
$\mathrm{MU}^{*}(X)$.} first Chern class $c_1^{\mathrm U}$. Lazard had shown
that the coefficient ring of the universal one-dimensional commutative formal
group law is a polynomial ring over $\mathbb Z$ \cite{Laz55}. Milnor and
Novikov obtained a similar polynomial description for
$\Omega_{\mathrm U}^{\ast}$ \cite{Mil60,Nov60,Nov62,Nov67}.

In the topology section of the 1966 International Congress of Mathematicians in Moscow, Novikov outlined a programme for the development of algebraic topology based on complex cobordism. During his visit to the United States in the spring of 1967, he presented a substantially expanded version of this programme, including applications of the formal group law arising in geometric cobordism. At a Harvard seminar attended by Quillen, Tate remarked after Novikov's lecture that he had never imagined algebraic topology to lie so close to his own mathematical interests \cite{Nov11}. Two years later, Quillen proved that the formal group law of complex cobordism is universal, or equivalently that the canonical classifying homomorphism from the Lazard ring to $\Omega_{\mathrm U}^{\ast}$ is an isomorphism \cite{Qui69}. This theorem explained the striking agreement between the polynomial descriptions that had arisen independently in formal group theory and in the computation of the complex cobordism ring. Novikov's study of geometric cobordism and Buchstaber's construction of the Chern--Dold character then provided effective topological tools for explicit calculations with the universal formal group law \cite{Nov60,Nov62,Nov67,Buc70}.

The classical realization question underlying this circle of ideas is the
Milnor--Hirzebruch problem, first formulated by Hirzebruch in his plenary
address at the 1958 International Congress of Mathematicians
\cite{Hir58}. Its algebro-geometric form asks which collections of the
$p(n)=|\mathcal P(n)|$ integers
\[
 c_I,\qquad I=(i_1,\ldots,i_r)\in\mathcal P(n),
\]
where $\mathcal P(n)$ is the set of partitions of $n$, can occur as the
ordinary tangent Chern numbers
\[
 c_I[M^{2n}]
 =\left\langle c_{i_1}(\tau_{M^{2n}})\cdots c_{i_r}(\tau_{M^{2n}}),[M^{2n}]\right\rangle
\]
of a smooth irreducible complex projective variety $M^{2n}$.\footnote{Equivalently,
the same realization question may be posed using ordinary tangent Chern
numbers, monomial symmetric tangent Chern numbers, or monomial symmetric
normal Chern numbers.} This realization problem remains largely
open. The Hirzebruch--Riemann--Roch theorem and its Grothendieck
generalization yield arithmetic congruences that every realizable collection
must satisfy, but they do not give a complete characterization.

Mishchenko found the following geometric expansion for the logarithm
of the universal formal group
\begin{equation}\label{eq:intro-mishchenko}
 g(u)=u+\sum_{n\geq1}\frac{[\CP^n]}{n+1}u^{n+1}.
\end{equation}
It is recorded in Appendix~1 of \cite{Nov67}. Buchstaber proved in
\cite[Theorem~2.3]{Buc70} that the exponential $f=g^{-1}$ of the universal formal group
has the Hurwitz expansion
\begin{equation}\label{eq:intro-theta-exponential}
 f(z)=z+\sum_{n\geq1}\frac{[M^{2n}]}{(n+1)!}z^{n+1}
\end{equation}
for certain stably complex manifolds $M^{2n}$ of real dimension $2n$.
Much later, Buchstaber and Veselov found smooth irreducible algebraic
representatives for these classes:
\[
 [M^{2n}]=[\Theta^n].
\]
Here \(\Theta^n\) is a smooth theta divisor of complex dimension \(n\) on a
general principally polarized abelian variety of dimension \(n+1\)
\cite[Theorem~1.1]{BV24}.

We distinguish primitivity, irreducibility, and undecomposability. A nonzero
element of a torsion-free abelian group is \emph{primitive} if it is not
divisible by any integer greater than one. A nonzero homogeneous class
$X\in\OmU$ of negative degree is \emph{irreducible} if
every factorization $X=YZ$ in $\OmU$ has a unit factor. Let
$\mathfrak a=\bigoplus_{n>0}\Omega_{\Uu}^{-2n}$ be the augmentation ideal.
We call $X$ \emph{undecomposable} if $X\notin\mathfrak a^2$, that is, if it is
not a finite sum of products of positive-dimensional cobordism classes. This
means that its image in $\mathfrak a/\mathfrak a^2$ is nonzero, not necessarily
primitive. In the graded polynomial ring $\OmU$, primitivity together with
undecomposability implies irreducibility. We say that a nonzero element \(W\in\Omega_U^*\otimes\mathbb Q\) has an exact denominator \(k\in\mathbb Z_{>0}\) if \(kW\in\Omega_U^*\) is primitive in the underlying additive group. It is clear that not every element admits an exact denominator (for example, \(\frac23[\mathbb CP^1]\) does not).

For every $n\geq1$, the classes $[\CP^n]$ and $[\Theta^n]$ are primitive
and undecomposable in $\Omega_{\Uu}^{-2n}$: their Todd genera are $1$ and
$(-1)^n$, respectively, while their Milnor numbers are $n+1$ and
$-(n+1)!$, respectively.

The passage from \eqref{eq:intro-theta-exponential} to
\eqref{eq:intro-mishchenko} is formally an inversion of power series, but its
arithmetic content is stronger than the general inversion theorem for Hurwitz
series. Recall that a strict Hurwitz series has the form
\[
 z+\sum_{n\geq1}a_n\frac{z^{n+1}}{(n+1)!},
\]
and its compositional inverse is again a Hurwitz series over the same
coefficient ring. This gives only the a priori factorial denominator bound
\((n+1)!\) for the coefficient of \(z^{n+1}\). In contrast,
\eqref{eq:intro-mishchenko} has denominator only \(n+1\). The Lagrange
inversion formula displays the cancellations responsible for this sharper bound:
\begin{equation}\label{eq:intro-lagrange}
 [\CP^n]
 =(n+1)\Ll_n\left(
 \frac{[\Theta^1]}{2!},\ldots,
 \frac{[\Theta^n]}{(n+1)!}
 \right),
\end{equation}
where \(\Ll_n\) is the \(n\)th Lagrange inversion polynomial
\cite{BV26}. More generally, Buchstaber and Veselov expressed every complex
cobordism class in terms of products of theta divisors. For any stably
complex manifold $X^{2n}$,
\begin{equation}\label{eq:intro-BV-normal-Chern}
 [X^{2n}]
 =\sum_{\lambda\vdash n}
 m_\lambda^\nu(X^{2n})\frac{[\Theta^\lambda]}{(\lambda+1)!},
\end{equation}
where, for $\lambda=(i_1,\ldots,i_k)$,
\[
 \Theta^\lambda=\Theta^{i_1}\times\cdots\times\Theta^{i_k},
 \qquad
 (\lambda+1)!=\prod_{j=1}^k(i_j+1)!,
\]
and $m_\lambda^\nu(X^{2n})$ is the monomial symmetric Chern number of the
stable normal bundle \cite[formula~(6)]{BV24}. The cancellation of the factor \(n!\) is therefore a genuine
divisibility statement in the integral complex cobordism ring.

Hirzebruch genera and formal groups provide effective methods for proving such
divisibility results. An integral genus converts divisibility of cobordism
classes into divisibility of characteristic numbers. A suitable genus can also
detect that a class is primitive. Classical divisibility theorems for the
Euler characteristic and the signature fit naturally into the same conceptual framework
\cite{Hir66}. A basic example is Ochanine's theorem that the signature of every
closed spin manifold of real dimension \(8k+4\) is divisible by \(16\). This
generalizes Rokhlin's theorem in dimension four and is proved by the methods of
Hirzebruch genera and formal groups \cite{Rok52,Och81}.

The two-valued formal group used in this paper comes from the
\emph{modulus square construction} of Buchstaber and Novikov \cite{BN71}.
We first recall its topological origin. Let $\zeta$ be the universal
quaternionic line bundle over $\HP^\infty$, let
$\zeta_i=\pr_i^*\zeta$ on $\HP^\infty\times\HP^\infty$, and let
$\iota\colon\CP^\infty\to\HP^\infty$ be induced by the inclusion
$S^1\subset\Sp(1)$. Let
$\rho\colon\Sp^*(-)\to\U^*(-)$ be the transformation that forgets the
quaternionic structure. If $\eta$ is the universal complex line bundle and
$\eta_i=\pr_i^*\eta$ on $\CP^\infty\times\CP^\infty$, then
$(\iota\times\iota)^*\zeta_i\cong\eta_i\oplus\bar\eta_i$ as complex bundles.

The complex tensor product $\zeta_1\otimes_{\mathbb C}\zeta_2$ is not a
quaternionic bundle. After pullback along $\iota\times\iota$, however, its underlying complex
bundle splits as
\[
 (\iota\times\iota)^*(\zeta_1\otimes_{\mathbb C}\zeta_2)
 \cong \xi_+\oplus\xi_-,
\]
where
\[
 \xi_+=\eta_1\eta_2\oplus\bar\eta_1\bar\eta_2,
 \qquad
 \xi_-=\eta_1\bar\eta_2\oplus\bar\eta_1\eta_2.
\]
Each $\xi_\pm$ has the form $L\oplus\bar L$ and carries its standard
quaternionic structure.

Let $F(u,v)$ be the universal formal group law in complex cobordism, put
\[
 u=c_1^{\U}(\eta_1),
 \quad \bar u=c_1^{\U}(\bar\eta_1),
 \qquad
 v=c_1^{\U}(\eta_2),
 \quad \bar v=c_1^{\U}(\bar\eta_2),
\]
and let $p_i^{\Sp}$ denote the Buchstaber--Novikov symplectic Pontryagin
classes. We use the normalization
$\rho\bigl(p_1^{\Sp}(L\oplus\bar L)\bigr)=c_2^{\U}(L\oplus\bar L)$.
For the universal bundle put
\[
 x_{\mathbb C}:=\rho\bigl(p_1^{\Sp}(\zeta)\bigr)
 =c_2^{\U}(\zeta_{\mathbb C})\in\U^4(\HP^\infty),
\]
where $\zeta_{\mathbb C}$ is the underlying complex bundle of rank two
associated with $\zeta$. Its restriction along $\iota$ satisfies
\[
 \begin{aligned}
 x:=\iota^*x_{\mathbb C}
 &=\iota^*c_2^{\U}(\zeta_{\mathbb C})
 =c_1^{\U}(\eta)c_1^{\U}(\bar\eta)
 \in\U^4(\CP^\infty),\\
 \iota^*c_1^{\U}(\zeta_{\mathbb C})
 &=c_1^{\U}(\eta)+c_1^{\U}(\bar\eta)\in\U^2(\CP^\infty).
 \end{aligned}
\]
By abuse of notation, on $\CP^\infty\times\CP^\infty$ we suppress the
projection pullbacks and write
$x=\pr_1^*(\iota^*x_{\mathbb C})$ and
$y=\pr_2^*(\iota^*x_{\mathbb C})$.
Then
\[
 \begin{aligned}
 x&=\rho\bigl(p_1^{\Sp}((\iota\times\iota)^*\zeta_1)\bigr)=u\bar u,
 &
 y&=\rho\bigl(p_1^{\Sp}((\iota\times\iota)^*\zeta_2)\bigr)=v\bar v,\\
 z_+&=\rho\bigl(p_1^{\Sp}(\xi_+)\bigr)
     =F(u,v)F(\bar u,\bar v),
 &
 z_-&=\rho\bigl(p_1^{\Sp}(\xi_-)\bigr)
     =F(u,\bar v)F(\bar u,v).
 \end{aligned}
\]
The Whitney formula now gives
\[
 \begin{aligned}
 \Theta_1(x,y)
  &=\rho\bigl(p_1^{\Sp}(\xi_+\oplus\xi_-)\bigr)=z_++z_-,\\
 \Theta_2(x,y)
  &=\rho\bigl(p_2^{\Sp}(\xi_+\oplus\xi_-)\bigr)=z_+z_-.
 \end{aligned}
\]
Thus $z_+$ and $z_-$ are the roots of the two-valued law
\begin{equation}\label{universal_two_valued_formal_group}
 z^2-\Theta_1(x,y)z+\Theta_2(x,y)=0.
\end{equation}
These formulas are the modulus square construction: they pass from the
one-valued law $F(u,v)$ to a two-valued law in the coordinates
$x=u\bar u$ and $y=v\bar v$. The elementary symmetric functions of the two
values $z_{+}$ and $z_{-}$ belong to $\Omega_{\mathrm U}^{\ast}[[x,y]]$, see \cite{BN71}. Every
one-dimensional commutative formal group law $F_A$ is classified by a homomorphism
$\varphi\colon\Omega_{\mathrm U}^{\ast}\to A$, and the modulus square
construction commutes with this coefficient change. Hence $\varphi$ carries
\eqref{universal_two_valued_formal_group} to the two-valued formal group
obtained from $F_A$. Thus every two-valued formal group arising from the
modulus square construction is a specialization of
\eqref{universal_two_valued_formal_group}. Buchstaber proved that, over the
subring generated by its coefficients, the law \eqref{universal_two_valued_formal_group} is isomorphic to the universal
two-valued formal group of the first type \cite[Theorems~2.3--2.4]{Buc78}.
We therefore call it the universal two-valued law. Its algebraic theory and
applications to cobordism were developed in \cite{Buc75,Buc76,Buc78}.
See also \cite{Buc06}.

Let \(g(u)\), \(f(u)\), and \(\bar u\) be the logarithm, exponential, and
formal inverse of \(F(u,v)\). Put
\(x=u\bar u\). Buchstaber and Novikov defined the logarithm \(B(x)\) of the
associated universal two-valued formal group by
\[
 B(x)=-g(u)^2.
\]
Let
\[
h:=c_2(\zeta_{\mathbb C})
 \in \mathrm H^4(\HP^\infty;\mathbb Z).
\]
The Chern--Dold character in symplectic cobordism is the natural
transformation \cite{Buc70}
\[
 \operatorname{ch}_{\Sp}\colon
 \Sp^*(X)\longrightarrow
 \mathrm H^*\!\left(X;\OmSp\otimes\mathbb Q\right).
\]
Here $\operatorname{ch}_{\U}$ is the Chern--Dold character in complex
cobordism introduced and explicitly calculated by Buchstaber in
\cite{Buc70}.
Applying the coefficient homomorphism induced by $\rho$ and using compatibility
of the Chern--Dold characters with $\rho$, Buchstaber and Novikov obtain
\cite[Lemma~2.16]{BN71}
\[
 (\mathrm{id}\otimes\rho)\operatorname{ch}_{\Sp}
 \bigl(p_1^{\Sp}(\zeta)\bigr)
 =\operatorname{ch}_{\U}(x_{\mathbb C})
 =B^{-1}(h).
\]
Here $\iota^*h=-z^2$ for $z=c_1(\eta)$, so restriction along $\iota$ gives
$$\operatorname{ch}_{\U}(x)=B^{-1}(-z^2)=f(z)f(-z).$$

The coefficients of \(B^{-1}(h)\) have canonical geometric numerators.
Buchstaber constructed classes \(D_n\in\Omega_{\Uu}^{-4n}\) satisfying
\begin{equation}\label{eq:intro-two-valued-exponential}
 B^{-1}(h)
 =h+\sum_{n\geq1}
 \frac{2(-1)^{n+1}D_n}{(2n+2)!}h^{n+1},
 \qquad
 \Td(D_n)=-1.
\end{equation}
Here \(\Td\) is the Todd genus whose exponential is
\(f_{\Td}(u)=1-e^{-u}\). Thus every \(D_n\) is primitive and undecomposable:
the Todd genus gives primitivity, while the Milnor number $s_{2n}$ vanishes on decomposable
classes and
\[
 s_{2n}(D_n)=(2n+2)!\ne0.
\]
The exact denominator of \([h^{n+1}]B^{-1}(h)\) is \((2n+2)!/2\)
\cite[Theorem~2.10]{Buc78}.

We recall the geometric representatives used in this result. For a tuple
$\omega=(n_1,\ldots,n_{2s})$ of nonnegative integers of even length, define the Cartesian product
\[
\CP_*^\omega
=
\prod_{\ell=1}^{2s}\CP_*^{2n_\ell+1},
\]
where \(\CP_*^{2n+1}\) denotes \(\CP^{2n+1}\) with Stong's stable
\(\Sp\)-structure
\[
\tau_{\CP_*^{2n+1}}\oplus\underline{\mathbb R}^{2}
\cong
(n+1)(\xi\oplus\bar\xi).
\]
Here \(\xi\to\CP^{2n+1}\) is the canonical complex line bundle. Let
\(\xi_\ell=\pr_\ell^{\ast}\xi\) on \(\CP_*^\omega\), and put
\[
\lambda_\omega
=
\xi_1\otimes\cdots\otimes\xi_{2s},
\qquad
\zeta_\omega
=
\lambda_\omega\oplus\bar\lambda_\omega.
\]
The bundle \(\zeta_\omega\) is a quaternionic line bundle. Explicitly, it has
the standard antilinear endomorphism
\(J(v,w)=(-\bar w,\bar v)\), which satisfies \(J^2=-1\). For
\(1\leq r\leq\sum_{\ell=1}^{2s}n_\ell+s\), the Stong manifold
\(M(\omega;r)\) is the zero locus of a generic transverse section of
\(\zeta_\omega^{\oplus r}\). Equivalently,
its fundamental class is Poincar\'e dual to
\(e(\zeta_\omega^{\oplus r})=c_{2r}(\zeta_\omega^{\oplus r})\). The tangent--normal bundle isomorphism
\[
 \tau_{M(\omega;r)}
 \oplus
 \left.\zeta_\omega^{\oplus r}\right|_{M(\omega;r)}
 \cong
 \left.\tau_{\CP_*^\omega}\right|_{M(\omega;r)}
\]
induces a stable \(\Sp\)-structure. The real dimension of \(M(\omega;r)\) is
\[
\dim_{\mathbb R}M(\omega;r)
=
4\left(\sum_{\ell=1}^{2s}n_\ell+s-r\right).
\]
The manifolds \(M(\omega;1)\) were introduced by Stong, who proved that their
classes multiplicatively generate \(\OmSp[1/2]\) \cite{Sto67}. Buchstaber
introduced and studied the manifolds \(M(\omega;r)\) for arbitrary positive
\(r\) \cite[\S~5, before Theorem~5.12]{Buc78}.

For every $1\leq r\leq\sum_{\ell=1}^{2s}n_\ell+s$,
Nadiradze constructed a transverse representative of this Stong cobordism
class that is invariant under a free involution $\Ii$ and whose quotient
\[
 M(\omega;r;\Ii)=M(\omega;r)/\Ii
\]
carries a self-conjugate stable complex structure \cite[pp.~264--265]{Nad80}. The notation
$M(\omega;r;\Ii)$ will always refer to such an invariant representative. The
quotient map is a two-sheeted covering compatible with the stable complex
structures, and Nadiradze's covering formula gives
\[
 [M(\omega;r)]=2[M(\omega;r;\Ii)]
 \qquad\text{in }\OmU.
\]

Let $\LambdaSt\subset\OmU$ be the subring generated by the classes
$[M(\omega;r;\Ii)]$. Buchstaber proved algebraically that the subring generated
by the halves $\tfrac12[M(\omega;r)]$ is already generated by
\[
 \frac12[M(k,l;q)],
 \qquad k,l\geq0,
 \qquad q=1,2,3
\]
\cite[Theorem~5.13]{Buc78}. Nadiradze's formula identifies these halves with the quotient
classes. Consequently, $\LambdaSt$ is generated by
\[
 [M(k,l;q;\Ii)],
 \qquad k,l\geq0,
 \qquad q=1,2,3.
\]

The zero tuple $0^{2n+2} = (0,\ldots, 0)$ gives the geometric representatives of $D_n$ in \eqref{eq:intro-two-valued-exponential}.
Indeed, put $\alpha=\alpha_1+\cdots+\alpha_{2n+2}$, where
$\alpha_i=c_1(\xi_i)$
on $A = (\CP_*^1)^{2n+2}$. The stable Chern classes of this product vanish.
Applying the universal characteristic series $Q(\alpha)=\alpha/f(\alpha)$ to the
tangent--normal formula and denoting $E=\lambda_{\omega_0}\oplus\overline{\lambda}_{\omega_0}$, $\omega_0 = 0^{2n+2}$ therefore gives
\[
[M(0^{2n+2};1)]
=
\left\langle
Q(\tau_A)\frac{c_2(E)}{Q(E)},
[A]
\right\rangle
=
\left\langle
\frac{-\alpha^2}{Q(\alpha)Q(-\alpha)},
[A]
\right\rangle
=
\left\langle f(\alpha)f(-\alpha),[A]\right\rangle.
\]
The last equality follows from the formula
$B^{-1}(-\alpha^2)=f(\alpha)f(-\alpha)$.
Nadiradze's covering formula then gives
\[
[M(0^{2n+2};1)]=2[M(0^{2n+2};1;\Ii)] = 2D_n,
\]
where $0^{2n+2}=(0,\ldots,0)$ has length $2n+2$.
Thus both the geometric numerators and the exact denominators of the universal
two-valued exponential $B^{-1}(h)$ are known.

In particular, $D_n=[M(0^{2n+2};1;\Ii)]$ belongs to $\LambdaSt$.
Recall that the Stong ring $\St\subset\OmU$ is characterized by the
vanishing of ordinary tangent Chern numbers\footnote{This criterion concerns
products of ordinary Chern classes, not the monomial symmetric numbers
$m_\lambda[X]=\langle m_\lambda(x_1,\ldots),[X]\rangle$. In particular,
$s_k=m_{(k)}$, whereas $c_k=m_{(1^k)}$.}
\[
 c_I[X]=\left\langle c_{i_1}\cdots c_{i_r},[X]\right\rangle=0
 \quad\text{whenever some }i_\nu\text{ is odd}.
\]
Here $I$ ranges over partitions of the complex dimension of $X$. Equivalently, if
$j\colon\OmSp/\Tors\to\OmU$ forgets the stable $\Sp$-structure, then
$\St$ is the $2$-primary saturation of $j(\OmSp/\Tors)$ in $\OmU$
\cite[Theorem~4.1 and Lemma~6.6]{Buc78}. Moreover, $\LambdaSt\subset\St$.

The coefficient rings associated with the universal two-valued law fit into a
useful hierarchy. Put
\[
 \Lambda
 =\mathbb Z[\text{coefficients of }\Theta_1\text{ and }
                  \text{coefficients of }\Theta_2]
 \subset\OmU
\]
and
\[
 \widetilde\Lambda
 =\mathbb Z\left[
    \text{coefficients of }\frac{\Theta_1}{2}\text{ and }
    \text{coefficients of }\Theta_2
  \right]
 \subset\OmU\!\left[\frac12\right].
\]
The latter is the ring denoted $\Lambda_1$ in the proof of
\cite[Theorem~4.4]{Buc78}.\footnote{Buchstaber denotes our rings $\St$ and
$\LambdaSt$ by a script $L$ and $\widetilde\Lambda$, respectively. Our $\widetilde\Lambda$ also differs from
the ring denoted by that symbol in \cite{Buc78}.}
If $I_{\Sp}=\operatorname{Im}(\OmSp\to\OmU)$, then Buchstaber's results give
\begin{equation}\label{eq:intro-ring-hierarchy}
\begin{aligned}
 \Lambda\subset\LambdaSt\subset\St
 &=(\widetilde\Lambda\cap\OmU)
 =\operatorname{Sat}_2(I_{\Sp}),\\
 \operatorname{Sat}_2(I_{\Sp})
 &=\{x\in\OmU\mid 2^rx\in I_{\Sp}\text{ for some }r\geq0\}.
\end{aligned}
\end{equation}
Indeed, $\St$ is a direct summand of $\OmU$ and
$\St[1/2]=I_{\Sp}[1/2]$ by \cite[Theorem~4.1 and Lemma~6.6]{Buc78}.
Theorem~4.4 of the same paper gives
$\St\subset\widetilde\Lambda\subset\St[1/2]$.
Intersecting with $\OmU$ yields both equalities in
\eqref{eq:intro-ring-hierarchy}.
Moreover, $\LambdaSt$ is generated by the coefficients of the three
series $H_1,H_2,H_3$ associated with the generalized shift (which we discuss
below). Thus $\Lambda$
records the integral coefficients of the universal two-valued law,
$\widetilde\Lambda$ controls the corresponding $2$-primary saturation, and
$\LambdaSt$ has a direct geometric description by quotient Stong manifolds.

The present paper treats the corresponding arithmetic problem for the
logarithm
\[
B(x)=x+\sum_{n\geq1}b_nx^{n+1}.
\]
Our main theorems,
\cref{thm:main,thm:stong-numerators,thm:sp-denominator,thm:Cn-Lambda}, show that for each
$n\geq1$,
\begin{equation}\label{denominator_d_n}
b_n=\frac{C_n}{d_n},
\qquad
d_n= \cfrac{n+1}{2} \lcm(1,2,\ldots,2n+2),
\end{equation}
where
\[
 C_n\in\Lambda^{-4n}\subset\LambdaSt^{-4n}\subset\St^{-4n}
 \subset\Omega_{\Uu}^{-4n}
\]
is primitive and undecomposable. More precisely, we compute the Milnor number
\[
 s_{2n}(C_n)=2(-1)^n d_n\ne0.
\]
Here $s_m$ is the characteristic number associated with the power sum
$\sum_i x_i^m$ in the tangent Chern roots.
Since $s_{2n}$ vanishes on $(\mathfrak a^2)^{-4n}$, its nonzero value proves
undecomposability, even after rationalization. Together with primitivity,
this also proves that every $C_n$ is irreducible in $\OmU$.

The assertion of primitivity is a substantive part of the theorem and is not
implied by the integrality of $C_n=d_nb_n$.
The argument for the upper bound only shows that $d_nb_n$ is an integral cobordism
class. A priori, this class could still be divisible by an integer greater
than one. The
proof of the exact denominator is divided into an integral upper bound and a
lower bound detected by a genus. The upper bound
follows from the Mishchenko formula, a ramified residue calculation, and a
parity property of the series $q(u)=-\bar u/u$. For the lower bound, we use
the integral index genus $\Phi_t$, recalled below, to detect the primitivity
of $C_n$.

It is useful to place this argument in relation to the Hattori--Stong theorem
\cite{Hat66,Sto65}. We first recall the relevant $\gamma$-operations. For
$E\in\mathrm{K}^0(X)$, put
\[
 \lambda_t(E)=\sum_{j\geq0}\lambda^j(E)t^j,
 \qquad
 \gamma_t(E)
 =\lambda_{t/(1-t)}(E)
 =\sum_{j\geq0}\gamma^j(E)t^j.
\]
Here the operations defined by exterior powers satisfy
\[
 \lambda^j([V])=\left[\bigwedge\nolimits^j V\right]
\]
for a complex vector bundle $V$ and are extended to virtual classes by
\[
 \lambda_t([V]-[W])
 =\frac{\lambda_t([V])}{\lambda_t([W])}.
\]
In particular, for a complex line bundle $L$,
\[
 \gamma_t([L]-1)=1+t([L]-1).
\]
If $M$ is a stably complex manifold of complex dimension $m$, let
\[
 \widetilde\tau_M=\tau_M-\mathbb C^m\in\widetilde{\mathrm{K}}^0(M).
\]
For $s\geq1$, define the multivariable Hattori--Stong polynomial
\[
 \Phi^{\mathrm{HS}}_{t_1,\ldots,t_s}(M)
 =\left\langle
   \Td(\tau_M)
   \prod_{\nu=1}^{s}
    \operatorname{ch}\gamma_{t_\nu}(\widetilde\tau_M),
   [M]
  \right\rangle.
\]
The Hattori--Stong theorem in terms of
$\Phi^{\mathrm{HS}}_{t_1,\ldots,t_s}$ states that, for
$\alpha\in\Omega_{\Uu}^{-2m}\otimes\Q$,
\begin{equation}\label{eq:intro-hattori-stong-phi}
 \alpha\in\Omega_{\Uu}^{-2m}
 \quad\Longleftrightarrow\quad
 [t_1^{i_1}\cdots t_s^{i_s}]
 \Phi^{\mathrm{HS}}_{t_1,\ldots,t_s}(\alpha)\in\mathbb Z
\end{equation}
for every $s\geq1$ and every tuple of nonnegative integers
$(i_1,\ldots,i_s)$ with $0\leq i_1+\cdots+i_s\leq m$. Thus the theorem characterizes
the subgroup $\Omega_{\Uu}^{-2m}$ inside its rationalization by the
integrality of these Hattori--Stong indices.

This criterion has an equivalent formulation. Let $\mathcal P_{\leq m}$ be the set of partitions
of weight at most $m$ (including the empty partition), and put
$A_m=\mathbb Z^{\mathcal P_{\leq m}}$.
For $\lambda=(\lambda_1,\ldots,\lambda_s)$ define
\[
 \kappa_\lambda(\alpha)
 =[t_1^{\lambda_1}\cdots t_s^{\lambda_s}]
 \Phi^{\mathrm{HS}}_{t_1,\ldots,t_s}(\alpha),
\]
with the empty partition giving the Todd genus, and let
$\kappa_m\colon \Omega_{\Uu}^{-2m}\to A_m$ collect these numbers. These indices separate
rational cobordism classes, and the Hattori--Stong theorem says precisely that
\begin{equation}\label{eq:intro-hattori-stong-saturated}
 \kappa_m(\Omega_{\Uu}^{-2m})
 =\kappa_{m,\mathbb Q}(\Omega_{\Uu}^{-2m}\otimes\mathbb Q)\cap A_m
 \qquad\text{inside }A_m\otimes\mathbb Q.
\end{equation}
Thus $\kappa_m(\Omega_{\Uu}^{-2m})$ is a saturated subgroup of the finitely generated free abelian group
$A_m$, hence a direct summand.

Buchstaber proved this direct summand formulation of the Hattori--Stong
theorem using formal group theory. Let
$\varphi(x)=x+\sum_{n\geq1}r_nx^{n+1}$ be the universal strict change of
coordinate. Applying it to the multiplicative formal group law
$x+y-\beta xy$ gives, over the polynomial $\mathbb Z$-algebra
$\mathbb Z[r_1,r_2,\ldots,\beta]$, the formal group law
\[
 F_{\varphi,\beta}(x,y)
 =\varphi^{-1}\!\left(\varphi(x)+\varphi(y)
          -\beta\varphi(x)\varphi(y)\right).
\]
He identified the corresponding classifying homomorphism from the Lazard ring
with the Hattori--Stong homomorphism and proved that its image is a direct
summand of $\mathbb Z[r_1,r_2,\ldots,\beta]$
\cite[Theorem~1.28 and \S7, Example~2]{Buc79Characteristic}.

For one variable, write
\[
 \Phi_t(M)=\Phi^{\mathrm{HS}}_t(M)
 =\sum_{j=0}^{m}\Phi_j(M)t^j.
\]
If $x_1,\ldots,x_m$ are the formal Chern roots of $\tau_M$, then the identity
for line bundles above gives
\[
 \operatorname{ch}\gamma_t(\widetilde\tau_M)
 =\prod_{i=1}^{m}\bigl(1+t(e^{x_i}-1)\bigr),
\]
and hence
\begin{equation}\label{eq:intro-Phi-roots}
 \Phi_t(M)
 =\left\langle
   \prod_{i=1}^{m}
    \frac{x_i}{1-e^{-x_i}}
    \bigl(1+t(e^{x_i}-1)\bigr),
   [M]
  \right\rangle.
\end{equation}
The full criterion \eqref{eq:intro-hattori-stong-phi} involves all
multivariable polynomials $\Phi^{\mathrm{HS}}_{t_1,\ldots,t_s}$. Its
specialization to $s=1$ is the polynomial $\Phi_t$, corresponding to a single
$\gamma$-operation.

We use this index genus for the lower bound. Its
characteristic series is
\[
 Q_{\Phi,t}(z)
 =\frac{z}{1-e^{-z}}\bigl(1+t(e^z-1)\bigr).
\]
The coefficients $\Phi_j(M)$ are integral by
\eqref{eq:intro-hattori-stong-phi}. We regard $\Phi_t$ as an ungraded genus
with values in $\mathbb Z[t]$.\footnote{Equivalently, introduce graded
variables $a,b$ with $\deg a=\deg b=-2$ and, for a stably complex manifold
$M$ of complex dimension $m$, set
$\Phi_{a,b}(M):=b^m\Phi_{a/b}(M)=\sum_{j=0}^m\Phi_j(M)a^jb^{m-j}$.
This is homogeneous of degree $-2m$ and defines a graded genus
$\Phi_{a,b}\colon\OmU\to\mathbb Z[a,b]$.} In complex dimension $m$, its
values have
degree at most $m$ in $t$. At $t=0$, it specializes to the Todd genus.

The coefficients $\Phi_j(C_n)$ of the polynomial $\Phi_t(C_n)$ are
Hattori--Stong numbers. The coefficientwise primitivity proved
below is therefore equivalent to
\[
 \gcd\bigl(\Phi_0(C_n),\Phi_1(C_n),\ldots,\Phi_{2n}(C_n)\bigr)=1.
\]
Equivalently, $d_n$ is the least common denominator of the rational numbers
\[
 \Phi_0(b_n),\Phi_1(b_n),\ldots,\Phi_{2n}(b_n).
\]
Thus this comparatively small subfamily of $\mathrm{K}$-theory indices already detects
the primitivity of $C_n$. It does not, however, replace the argument for the
upper bound: an application of the full Hattori--Stong criterion to the inclusion
$d_nb_n\in\Omega_{\Uu}^{-4n}$ would require the integrality of all coefficients
of all the multivariable polynomials
$\Phi^{\mathrm{HS}}_{t_1,\ldots,t_s}$. The Chern number theorem below shows why
the Hattori--Stong indices are needed: all ordinary Chern numbers of the rational
class $b_n$ are integral, although $b_n$ itself does not belong to the subgroup
$\Omega_{\Uu}^{-4n}\subset\Omega_{\Uu}^{-4n}\otimes\Q$.

The geometric inclusion $C_n\in\LambdaSt$ is proved by a different argument,
based on Buchstaber's generalized shift series $H_q(x,y)$. Buchstaber's
coefficient formula and Nadiradze's quotient construction give, respectively,
\[
 2H_q(x,y)
 =-\sum_{k,l\geq0}[M(k,l;q)]x^ky^l,
 \qquad
 H_q(x,y)
 =-\sum_{k,l\geq0}[M(k,l;q;\Ii)]x^ky^l.
\]
For fixed $n$, the values
\[
 P_n(q)=[x^{q-1}y^n]H_q(x,y),
 \qquad q=1,2,\ldots,
\]
extend to an odd integer-valued polynomial of degree at most $2n+1$, and a
residue calculation gives $P_n'(0)=(n+1)b_n$. An arithmetic lemma for odd
integer-valued polynomials then yields $d_nb_n\in\LambdaSt$. In fact, the
argument expresses $C_n$ as an integral linear combination of the specific
quotient Stong manifolds $M(q-1,n;q;\Ii)$, with $1\leq q\leq n+1$:
\begin{equation}\label{eq:Cn-stong-combination-intro}
 C_n=-\sum_{q=1}^{n+1}A_{n,q}[M(q-1,n;q;\Ii)],
\end{equation}
where
\[
 A_{n,q}
 =(-1)^{q+1}\frac{2d_n}{(n+1)q}
   \frac{\binom{2n+2}{n+1-q}}{\binom{2n+2}{n+1}}
 \in\mathbb Z,
 \qquad 1\leq q\leq n+1.
\]

The integers
\[
 e_n=\frac12\lcm(1,2,\ldots,2n+2)
\]
that occur in this argument are not auxiliary bookkeeping constants.
\Cref{lem:odd-integer-valued} identifies $e_n$ as the exact universal
denominator of differentiation at the origin on the group
\[
 \operatorname{Int}^{\mathrm{odd}}_{\leq2n+1}(\mathbb Z)
 =\{P\in\mathbb Q[T]\mid P(-T)=-P(T),\ 
      \deg P\leq2n+1,\ P(\mathbb Z)\subset\mathbb Z\}.
\]
More precisely,
\[
 e_n
 =\min\left\{e\geq1\ \middle|\
   eP'(0)\in\mathbb Z
   \text{ for every }P\in
   \operatorname{Int}^{\mathrm{odd}}_{\leq2n+1}(\mathbb Z)
  \right\},
\]
and the derivative is recovered integrally from the first $n+1$ positive
values by
\[
 e_nP'(0)=\sum_{q=1}^{n+1}A_{n,q}P(q).
\]
For the geometric polynomial $P_n$, one has
$P_n'(0)=(n+1)b_n$. Hence
$d_n=(n+1)e_n$. This separates the factor $n+1$ coming from Lagrange inversion
from the optimal interpolation denominator $e_n$.

The Stong manifold representation \eqref{eq:Cn-stong-combination-intro} also
gives a quaternionic refinement of the logarithm. Put
\[
 \OmSpfree^{\ast}:=\OmSp/\Tors.
\]
Novikov proved that $\Omega_{\mathrm{Sp}}^{\ast}$ has no odd-primary torsion
\cite[Theorem~4, p.~428]{Nov62}. Consequently all its torsion is $2$-primary,
and the quotient map induces a canonical isomorphism
\[
 \Omega_{\mathrm{Sp}}^{\ast}\!\left[\frac12\right]
 \xrightarrow{\ \cong\ }
 \widetilde{\Omega}_{\mathrm{Sp}}^{\ast}\!\left[\frac12\right].
\]
Let $j\colon\OmSpfree^{\ast}\longrightarrow\OmU$ be the homomorphism
induced by forgetting the stable $\Sp$-structure. Its rational extension $j_{\Q}$
identifies $\OmSpfree^{\ast}\otimes\Q$ with $\St\otimes\Q$. Since
$b_n=C_n/d_n\in\St\otimes\Q$, there is a unique quaternionic logarithm
\[
 \widetilde B(x)
 =j_{\Q}^{-1}(B(x))
 =x+\sum_{n\geq1}\widetilde b_nx^{n+1}
 \in(\OmSpfree^{\ast}\otimes\Q)[[x]].
\]
Equivalently, after passage to $\OmSpfree^{\ast}\otimes\Q$, the Chern--Dold
character of $p_1^{\Sp}(\zeta)\in\Sp^4(\HP^\infty)$ is
$\widetilde B^{-1}(h)$, where $h=c_2^H(\zeta_{\mathbb C})$ as above. For
$n\geq1$, put
\[
 \widetilde{d_n}
 :=\min\{N\geq1\mid N\widetilde b_n\in\OmSpfree^{-4n}\},
 \qquad
 \widetilde C_n:=\widetilde{d_n}\widetilde b_n.
\]
The double covers lift \eqref{eq:Cn-stong-combination-intro} to the integral
quaternionic class
\[
 \widehat C_n
 :=-\sum_{q=1}^{n+1}A_{n,q}[M(q-1,n;q)]_{\Sp}
 \in\OmSpfree^{-4n}.
\]
By \cref{prop:sp-covering-numerator}, it satisfies
$\widehat C_n=2d_n\widetilde b_n$ and $j(\widehat C_n)=2C_n$.
\Cref{thm:sp-denominator} determines the exact denominator:
\[
 \widetilde{d_n}=2d_n
 =(n+1)\lcm(1,\ldots,2n+2)
 \qquad(n\geq1).
\]
In particular, $\widetilde C_n=\widehat C_n$ is primitive in
$\OmSpfree^{-4n}$, and $C_n\notin j(\OmSpfree^{-4n})$ for every $n\geq1$.

To exclude an integral quaternionic lift for $C_n$, we use the
Buchstaber--Novikov symplectic Pontryagin classes and the
Riemann--Roch--Grothendieck--Hirzebruch formula in complex cobordism
\cite[Theorem~2.4, formula~(29)]{BV24}.
A hypothetical lift would express the coefficients of $\Phi_t(C_n)$ in
$c=t(1-t)$, up to sign, as ordinary Todd genera of integral quaternionic
cobordism classes. Their evenness in dimensions $8k+4$
\cite[Corollary~2(ii), p.~279]{AH59} contradicts an explicit odd coefficient.
The proof is given in \cref{sec:sp-lift}.

The passage from $\LambdaSt$ to the smaller ring $\Lambda$ is entirely
$2$-primary, because
\[
 \LambdaSt\subset\St\subset\widetilde\Lambda\subset\Lambda\!\left[\frac12\right].
\]
To recover the missing dyadic information, we restrict $\Theta_1$ to the
diagonal $y=x$ and use the identities
\[
 \Theta_1(x,x)=B^{-1}(4B(x)),
 \qquad
 B\bigl(\Theta_1(x,x)\bigr)=4B(x).
\]
The unit and symmetry relations for $\Theta_1$ imply that every coefficient
of an odd power $x^k$, $k>1$, in $\Theta_1(x,x)$ is divisible by $2$ in
$\Lambda$. An estimate over cyclic orbits for the coefficients of $\Theta_1(x,x)^m$,
followed by induction using Schröder's functional
equation~\cite[\S~7]{MNT17} $B\bigl(\Theta_1(x,x)\bigr)=4B(x)$, gives
\[
 2^{\nu_2(d_n)}b_n\in\Lambda\otimes\Ztwo,
 \qquad
 \nu_2(d_n)=\nu_2(n+1)+\left\lfloor\log_2(n+1)\right\rfloor.
\]
Together with $C_n\in\Lambda[1/2]$, this proves
\[
 C_n\in\Lambda^{-4n}\qquad\text{for every }n\geq1.
\]
In fact, $\Lambda\subsetneq\LambdaSt$ by \cite[Corollary~5.8]{Buc78}.
The theorem identifies the distinguished sequence of primitive logarithmic
numerators inside the smaller coefficient ring.

The index genus $\Phi_t$ detects the full denominator $d_n$, whereas odd
genera need not do so. More precisely,
\cref{thm:bc-odd-denominator} shows that the genus $\Bc_n$ detects only the
odd part of the number $d_n$. For $N\geq1$, the genus $\Bc_N$ corresponds to
a formal group $F_{\Bc,N}(u,v)$ with odd exponential $f_{\Bc,N}(u)$ satisfying
\[
 \bigl(f_{\Bc,N}'(z)\bigr)^2
 =1-a_1f_{\Bc,N}(z)^2+a_2f_{\Bc,N}(z)^4-
 \cdots+(-1)^Na_Nf_{\Bc,N}(z)^{2N}.
\]
For $N=3$, this is the Buchstaber genus $\Bc$ studied in \cite{Kor26}. We
regard the genera $\Bc_N$ as truncations of the universal odd genus
$\Bc_\infty$, whose odd exponential satisfies
\[
 \bigl(f_{\Bc,\infty}'(z)\bigr)^2
 =1-a_1f_{\Bc,\infty}(z)^2+a_2f_{\Bc,\infty}(z)^4-
 a_3f_{\Bc,\infty}(z)^6+\cdots.
\]
The genus $\Bc_\infty$ is universal among odd genera. Its associated two-valued
law, obtained by the modulus square construction, is the coefficientwise image under $\Bc_\infty$ of
\eqref{universal_two_valued_formal_group}. Since the coefficients of
$B^{-1}$ and $B$ belong to $\St\otimes\mathbb Q$, it is natural to consider
the restrictions of these genera to the Stong ring. We prove the general
$2$-local statement
\[
 \Bc_\infty(\St)
 \subset\Ztwo[a_1,a_2,\ldots],
\]
where $\Ztwo$ is the ring of rational numbers with odd denominator. Thus
neither the universal odd genus nor any of its finite truncations has a
$2$-primary denominator on $\St$. The genera $\Bc_1$, $\Bc_2$, and
$\Bc_3$ ($=\Bc$) are in fact integral
on $\St$, and for the $\Bc_3$ genus one has
\[
 \Bc_3(\St)\subset\mathbb Z[2a_1,a_2,2a_3].
\]
For every $N>3$, however, $\Bc_N$ is not integral on $\St$. Indeed, if
\[
 \theta_{32}=[x^3y^2]\Theta_1(x,y)\in\St,
\]
then the specialization in which all parameters except $a_4$ vanish gives
\[
 \left.\Bc_N(\theta_{32})\right|_{a_j=0,\ j\ne4}
 =\frac{140}{3}a_4.
\]
Thus $\Bc_3$ is the last integral truncation of the universal odd genus. See
\cref{thm:bc-stong-threshold}.

There is a general reason why odd genera are sufficient for computations on
\(\St\). If \(Q(z)\) is the characteristic series of a Hirzebruch genus, put
\[
 \widetilde Q(z)=\sqrt{Q(z)Q(-z)},
 \qquad \widetilde Q(0)=1.
\]
We prove in \cref{thm:oddification-stong} that the genera determined by
\(Q\) and \(\widetilde Q\) have the same restriction to the Stong ring. In
particular, the specializations of the $\Bc_3$ genus at
\[
 (a_1,a_2,a_3)=(-3,3,-1)
 \quad\text{and}\quad
 (a_1,a_2,a_3)=\left(-\frac14,0,0\right)
\]
coincide on \(\St\) with the top Chern number genus \footnote{On almost
complex representatives the top Chern number agrees with the topological
Euler characteristic. For an arbitrary stable complex structure it need not.}
$\chi$ and the Todd genus,
respectively. The proof uses the oriented extension determined by
$Q(z)Q(-z)$: on quaternionic representatives the Chern roots occur in pairs
$\pm x_i$, so the common characteristic class is a power series in the
Pontryagin roots $x_i^2$. These specializations do not exhaust the information
contained in the $\Bc_3$ genus. For every $n\geq5$, we construct classes
$X_n,Y_n\in\Lambda^{-4n}\subset\St^{-4n}$ such that $X_n$, $Y_n$, and
$X_n-Y_n$ are all undecomposable in $\OmU$, even after rationalization, and
\[
 \Och(X_n)=\Och(Y_n),\qquad \chi(X_n)=\chi(Y_n),
 \qquad \Bc(X_n)\ne\Bc(Y_n).
\]
The construction uses the coefficients of $\Theta_1$ directly. See
\cref{thm:bc-not-determined}.

The integrality question also extends beyond the Stong ring. The $\Bc_3$
genus defines an oriented genus after inverting $2$. We prove that its value
on every closed spin manifold of real dimension less than $24$ already lies
in $\mathbb Z[a_1,a_2,a_3]$, see \cref{thm:bc-spin-below24}. The proof compares
it with the Ochanine genus and eliminates the possible halves of
characteristic numbers by the Wu formulas. In particular, the same
integrality holds for compact Calabi--Yau manifolds\footnote{By Calabi--Yau manifolds we mean stably complex manifolds $M$ with $c_1(\tau_M) = 0$ in $H^2(M;\mathbb Z)$.} in these dimensions.
We further prove that the genus is integral on every closed string
$24$-manifold by evaluating its degree-$24$ characteristic polynomial on the
integral basis of string bordism constructed by Han and Huang
\cite[Theorem~1 and Corollary~3]{HH22String}. Thus $\Bc_3$ is integral on
string manifolds through real dimension $24$, see
\cref{thm:bc-string-through24}. The spin dimension bound is sharp: the
Anderson--Brown--Peterson example in \cref{ex:abp-spin24} has real dimension
$24$ and nonintegral $\Bc_3$ genus.

The classes $C_n$ lead to a particularly concrete geometric realization
problem.  The Stong manifold formula
\eqref{eq:Cn-stong-combination-intro} represents $C_n$ by an integral combination of smooth self-conjugate
manifolds in complex cobordism.  One would like instead to
represent $C_n$ by a single connected almost complex manifold, and preferably
by a smooth irreducible projective variety.  Such representatives exist in the
first two degrees.  Let $S$ be an Enriques surface.  Then
\begin{equation}\label{eq:C1-C2-Enriques}
 C_1=[S],\qquad
 C_2=[S^{[2]}]\in\OmU,
 \qquad S^{[2]}=\operatorname{Hilb}^{2}(S).
\end{equation}
Indeed, the canonical class of $S$ is $2$-torsion and
\[
 c_1^2(S)=0,\qquad c_2(S)=\chi(S)=12=d_1,
 \qquad \Td(S)=1
\]
\cite{BHPV04}.  The universal formulas for Hilbert schemes of points,
together with G\"ottsche's Euler characteristic formula, give
\[
 c_1\bigl(S^{[2]}\bigr)=0\ \text{in rational cohomology},\qquad
 c_4\bigl(S^{[2]}\bigr)=\chi\bigl(S^{[2]}\bigr)=90=d_2,
\]
\[
 \Td\bigl(S^{[2]}\bigr)=1,\qquad
 c_2^2\bigl(S^{[2]}\bigr)=270
\]
\cite{EGL01,Got90}.  On the other hand,
\eqref{eq:coefficient-B-Phi} at $t=0$ gives
$\Td(C_1)=\Td(C_2)=1$, while \cref{thm:chern-divisibility} gives
$c_2(C_1)=12$ and $c_4(C_2)=90$.  Since $C_1,C_2\in\St$, all their ordinary Chern numbers containing
$c_1$ vanish. The Todd formula in complex dimension four then gives
$c_2^2(C_2)=270$.  Thus the Chern numbers on the two sides of
\eqref{eq:C1-C2-Enriques} coincide, and the Milnor--Novikov theorem proves the
claimed equalities.

For general $n$, \cref{thm:chern-divisibility} shows that every monomial
symmetric Chern number of the primitive class $C_n$ is divisible by $d_n$ and
that
\[
 c_{2n}(C_n)=d_n.
\]

The Hilbert scheme construction just considered cannot be prolonged:
G\"ottsche's formula already gives
\[
 \chi\bigl(S^{[3]}\bigr)=520\ne 1680=d_3=c_6(C_3),
\]
so the pattern $C_n=[S^{[n]}]$ breaks at $n=3$ and cannot be prolonged to a
family valid for all $n\geq3$.

We leave open the problem of finding comparably natural connected almost
complex representatives $\mathfrak C_n$ for $n\geq3$, preferably smooth
irreducible projective varieties.
Any such representative would satisfy
\[
 \chi(\mathfrak C_n)=d_n,
 \qquad
 \gcd\{m_\lambda[\mathfrak C_n]\mid\lambda\vdash2n\}=d_n.
\]
This is a distinguished family of cases of the Milnor--Hirzebruch realization
problem.  It is also related to the question of Buchstaber and Veselov
\cite{BV26} asking for a characterization of smooth irreducible complex
projective varieties $M^m$ with $\chi(M^m)\ne0$ for which
\[
 \gcd\{m_\lambda[M^m]\mid\lambda\vdash m\}=|\chi(M^m)|.
\]
Complex projective spaces and smooth theta divisors provide basic examples of
this equality.

The paper is organized as follows. \Cref{sec:statement} collects the main results: the exact denominator theorem; the inclusion of the numerator classes in $\LambdaSt$; the exact quaternionic denominators and primitive numerators; the inclusion of every numerator class $C_n$ in the coefficient ring
$\Lambda$; applications of the odd genera $\Bc_N$; the divisibility of the Chern numbers of $C_n$; the oddification theorem; spin integrality below dimension $24$; and string integrality through dimension $24$.
\Cref{sec:main-proof} proves the exact denominator theorem. The upper bound
uses Mishchenko's formula, a ramified residue formula, and a $2$-divisibility
lemma, while the lower bound is detected by the integral index genus $\Phi_t$.
The final Milnor number calculation proves undecomposability and
irreducibility of $C_n$.
\Cref{sec:stong-numerators} first recalls the topological origin of the
invariant differentiation and generalized shift, then studies the series
$H_q$, proves their closed formula, and uses odd integer-valued interpolation
to prove that $C_n\in\LambdaSt$ and to obtain explicit Stong manifold
expressions for the classes $C_n$. Lifting these expressions along the double
covers gives integral quaternionic classes. Their primitivity and the exact
formula $\widetilde{d_n}=2d_n$ follow from the parity obstruction from the Todd genus.
The section then uses $\Theta_1(x,x)$ to prove that every $C_n$ lies in the
coefficient ring $\Lambda$.
\Cref{sec:chern-divisibility} proves the
Chern number divisibility theorem by means of a universal genus recording all
Chern numbers. In \cref{sec:odd-genera} we prove the $2$-local
integrality of the universal odd genus, show that $\Bc_N$ gives the exact odd
denominators for $\Bc_N(b_n)$, determine the integral truncations on $\St$,
prove the oddification theorem, and show that the Ochanine genus together with
the top Chern number genus does not determine the $\Bc_3$ genus even on
undecomposable classes in $\Lambda$. Finally, \cref{sec:spin-integrality}
proves integrality of $\Bc_3$ on spin manifolds of real dimension less than
$24$ and on string manifolds of real dimension at most $24$.

\vspace{0.5em}

\noindent\textbf{Financial support.}
Support from the Basic Research Program of HSE University is gratefully
acknowledged (HSE-BR-2025-060).

\section{Main results}\label{sec:statement}

\subsection{Exact denominators and primitive numerators}

The central problem is to clear the denominators of the coefficients of
$B(x)$ by the smallest possible integers.

Let
\[
 F(u,v)\in\OmU[[u,v]]
\]
be the universal one-dimensional formal group law.  We write
\[
 g(u)=u+O(u^2),\qquad f(z)=z+O(z^2)
\]
for its logarithm and exponential, so that $f=g^{-1}$.  Its formal inverse is
\[
 \bar u=[-1]_F(u)=f(-g(u)).
\]
The logarithm of the associated universal two-valued formal group is the
unique strict series
\[
 B(x)=x+\sum_{n\geq1}b_nx^{n+1}
     \in (\OmU\otimes\mathbb{Q})[[x]]
\]
satisfying
\begin{equation}\label{eq:def-B}
 B(u\bar u)=-g(u)^2.
\end{equation}
Equivalently, its exponential $B^{-1}(h)$ is characterized by
\begin{equation}\label{eq:def-exp-B}
 B^{-1}(-z^2)=f(z)f(-z).
\end{equation}
For $n\geq1$, define

\begin{equation}\label{eq:def-dn}
\begin{aligned}
d_n = \cfrac{n+1}{2} \lcm(1,2,\ldots,2n+2).
 \end{aligned}
\end{equation}

\begin{theorem}[(The exact denominator theorem)]\label{thm:main}
For every $n\geq1$, there is a primitive and undecomposable class
$C_n\in\Omega_{\Uu}^{-4n}$ such that
\begin{equation}\label{b_n_as_fraction}
 b_n=\frac{C_n}{d_n}.
\end{equation}
Its Milnor number is
\[
 s_{2n}(C_n)=2(-1)^n d_n.
\]
In particular, $C_n$ is irreducible in $\OmU$, and $d_n$ is the exact
denominator of $b_n$ in $\OmU\otimes\mathbb{Q}$.
\end{theorem}

The proof is given in \cref{sec:main-proof}, with the final argument in \cref{sec:lower}.

The denominator statement follows from an integral upper bound and a lower
bound detected by the index genus. A Milnor number calculation proves
undecomposability.

\subsection{Stong manifold representatives of the numerators}

Let $\LambdaSt\subset\OmU$ be the subring generated by the quotient Stong
manifolds $[M(\omega;q;\Ii)]$. For a two-term tuple we write
$M(k,l;q;\Ii)=M((k,l);q;\Ii)$.

\begin{theorem}[(The Stong manifold numerator theorem)]
\label{thm:stong-numerators}
For every $n\geq1$, the primitive numerator $C_n=d_nb_n$ belongs to
$\LambdaSt^{-4n}$. More precisely, put
\[
 L_{2n+2}=\lcm(1,2,\ldots,2n+2),
 \qquad
 e_n=\frac{L_{2n+2}}2,
\]
and, for $1\leq q\leq n+1$, put
\begin{equation}\label{eq:Anq-main}
 A_{n,q}
 =(-1)^{q+1}\frac{2e_n}{q}
   \frac{\binom{2n+2}{n+1-q}}{\binom{2n+2}{n+1}}.
\end{equation}
Then $A_{n,q}\in\mathbb Z$ and
\begin{equation}\label{eq:Cn-stong-combination-main}
 C_n
 =-\sum_{q=1}^{n+1}
 A_{n,q}[M(q-1,n;q;\Ii)].
\end{equation}
In particular,
\[
 C_n\in\LambdaSt^{-4n}
 \qquad\text{for every }n\geq1.
\]
\end{theorem}

The proof is given in \cref{sec:odd-polynomial-arithmetic}.

The dimension formula gives
\[
 \dim_{\mathbb R}M(q-1,n;q;\Ii)=4n,
\]
so the representation in \eqref{eq:Cn-stong-combination-main} is homogeneous.

Let $H_q(x,y)$ be the generalized shift series defined in
\cref{sec:stong-numerators}. If
\[
 R_H
 =\mathbb Z[\text{coefficients of }H_1,H_2,H_3],
\]
then the quotient expansion and generator statement in
\cref{prop:Hq-stong-expansion,cor:Hq-recurrence-generators} give
\begin{equation}\label{eq:RH-LambdaSt-main}
 R_H=\LambdaSt.
\end{equation}
Thus \cref{thm:stong-numerators} also proves that every $C_n$ is an integral
polynomial in the coefficients of $H_1,H_2,H_3$.

\subsection{The quaternionic logarithm}

Put
\begin{equation}\label{eq:def-Omega-Sp-free}
 \OmSpfree^{\ast}:=\OmSp/\Tors.
\end{equation}
Let
\[
 j\colon\OmSpfree^{\ast}\longrightarrow\OmU
\]
be induced by forgetting the stable $\Sp$-structure. By Stong's theorem
\cite[\S4]{Sto67}, the forgetful homomorphism
\[
\Omega_{\Sp}^{*}[1/2]\longrightarrow\Omega_{\mathrm {SO}}^{*}[1/2]
\]
is an isomorphism. Since it factors through $j[1/2]$, the latter is
injective. On the other hand, \eqref{eq:intro-ring-hierarchy} identifies
its image with $\St[1/2]$. Hence
\[
j[1/2]\colon
\widetilde{\Omega}_{\Sp}^{*}[1/2]
\xrightarrow{\cong}
\St[1/2],
\]
and therefore its rational extension is an isomorphism
\[
j_{\Q}\colon
\widetilde{\Omega}_{\Sp}^{*}\otimes\Q
\xrightarrow{\cong}
\St\otimes\Q.
\] Since \cref{thm:stong-numerators} gives $b_n\in\St\otimes\Q$, define the
quaternionic logarithm coefficientwise by
\begin{equation}\label{eq:def-B-tilde}
 \widetilde B(x)
 :=j_{\Q}^{-1}(B(x))
 =x+\sum_{n\geq1}\widetilde b_nx^{n+1}
 \in(\OmSpfree^{\ast}\otimes\Q)[[x]].
\end{equation}
Equivalently, after passage to $\OmSpfree^{\ast}\otimes\Q$, the Chern--Dold
character of $p_1^{\Sp}(\zeta)\in\Sp^4(\HP^\infty)$ is
$\widetilde B^{-1}(h)$, where $h=c_2^H(\zeta_{\mathbb C})$ as above. For
$n\geq1$ put
\begin{equation}\label{eq:def-dn-Sp}
 \widetilde{d_n}
 :=\min\{N\geq1\mid N\widetilde b_n\in\OmSpfree^{-4n}\}.
\end{equation}
Using the integers $A_{n,q}$ from \eqref{eq:Anq-main}, define
\begin{equation}\label{eq:def-Cn-hat}
 \widehat C_n
 :=-\sum_{q=1}^{n+1}
 A_{n,q}[M(q-1,n;q)]_{\Sp}
 \in\OmSpfree^{-4n},
\end{equation}
where the subscript indicates the quaternionic cobordism class.

\begin{proposition}
\label[proposition]{prop:sp-covering-numerator}
For every $n\geq1$,
\begin{equation}\label{eq:Cn-hat-images}
 j(\widehat C_n)=2C_n,
 \qquad
 \widehat C_n=2d_n\widetilde b_n.
\end{equation}
\end{proposition}

Consequently,
\begin{equation}\label{eq:dn-Sp-dichotomy}
 \widetilde{d_n}\in\{d_n,2d_n\}.
\end{equation}
The second alternative always holds.

\begin{theorem}[(The quaternionic denominator theorem)]
\label{thm:sp-denominator}
We have
\begin{equation}\label{eq:Sp-lift-equivalences}
 \widetilde{d_n}=2d_n
 =(n+1)\lcm(1,\ldots,2n+2).
\end{equation}
\end{theorem}

In particular, $C_n\notin j(\OmSpfree^{-4n})$, and the exact numerator
\[
 \widetilde C_n:=\widetilde{d_n}\widetilde b_n
 =\widehat C_n\in\OmSpfree^{-4n}
\]
is primitive.

The proofs are given in \cref{sec:sp-lift}.

Thus the explicit covering class $\widehat C_n$ is the primitive quaternionic
numerator in every degree. The case $n=1$ is illustrated in
\cref{rem:first-Sp-denominator}.

\subsection{The coefficient ring \texorpdfstring{$\Lambda$}{Lambda}}

Let
\[
 \Lambda
 =\mathbb Z[\text{coefficients of }\Theta_1\text{ and }
                  \text{coefficients of }\Theta_2]
 \subset\OmU
\]
and
\[
 \widetilde\Lambda
 =\mathbb Z\left[
    \text{coefficients of }\frac{\Theta_1}{2}\text{ and }
    \text{coefficients of }\Theta_2
  \right]
 \subset\OmU\!\left[\frac12\right].
\]
Buchstaber's results
\cite[Theorems~4.1, 4.4, 5.13 and Lemma~6.6]{Buc78} give
\begin{equation}\label{eq:ring-hierarchy-main}
 \Lambda\subset\LambdaSt\subset\St,
 \qquad
 \St=\widetilde\Lambda\cap\OmU.
\end{equation}
The equality $R_H=\LambdaSt$ in \eqref{eq:RH-LambdaSt-main} gives the geometric
form of the middle ring. Since $C_n$ is primitive in $\OmU$, its membership in
$\LambdaSt$ also implies that it is primitive in both $\LambdaSt$ and $\St$.

\begin{theorem}[(Numerators in the coefficient ring $\Lambda$)]
\label{thm:Cn-Lambda}
For every $n\geq1$, one has
\[
 C_n\in\Lambda^{-4n}.
\]
In particular, $C_n$ is primitive in $\Lambda$.
\end{theorem}

The proof is given in \cref{sec:coefficient-ring}. Its $2$-primary input is the
diagonal identity
\[
 \Theta_1(x,x)=B^{-1}(4B(x)).
\]

\subsection{Odd genera and their restrictions}

The finite odd genera $\Bc_N$ provide natural specializations of the universal
two-valued logarithm. They detect the odd-primary parts of the denominators of
the coefficients $b_n$ (see \eqref{b_n_as_fraction}), and their restrictions to the Stong ring exhibit a
sharp change at $N=3$.

Let $f_{\Bc,\infty}(z)$ be the universal odd series \(f_{\mathrm{Bc},\infty}(z)=z+O(z^3)\) satisfying
\begin{equation}\label{eq:bc-infty-exponential}
 \bigl(f_{\Bc,\infty}'(z)\bigr)^2
 =1-a_1f_{\Bc,\infty}(z)^2+a_2f_{\Bc,\infty}(z)^4-
  a_3f_{\Bc,\infty}(z)^6+\cdots,
\end{equation}
and let
\[
 \Bc_\infty\colon\OmU\longrightarrow\mathbb Q[a_1,a_2,\ldots]
\]
be the corresponding genus. We write $\Ztwo$ for the localization of
$\mathbb Z$ at the prime ideal $(2)$, that is, the ring of rational numbers
with odd denominator.

\begin{theorem}[(Two-local integrality of the universal odd genus $\Bc_{\infty}$)]
\label{thm:bc-infty-two-local}
One has
\[
 \Bc_\infty(\St)\subset\Ztwo[a_1,a_2,\ldots].
\]
Consequently, no finite truncation $\Bc_N$ has a $2$-primary denominator on
$\St$.
\end{theorem}

The proof is given in \cref{sec:two-local-integrality}.

For an integer $N\geq1$, let $f_{\Bc,N}(z)$ be the unique odd strict series
satisfying
\begin{equation}\label{eq:bcN-exponential}
 \bigl(f_{\Bc,N}'(z)\bigr)^2
 =1-a_1f_{\Bc,N}(z)^2+a_2f_{\Bc,N}(z)^4-
 \cdots+(-1)^Na_Nf_{\Bc,N}(z)^{2N}.
\end{equation}
It defines a genus
\[
 \Bc_N\colon\OmU\longrightarrow\mathbb Q[a_1,\ldots,a_N],
 \qquad \deg a_j=-4j.
\]
For $N=3$, this is the Buchstaber genus $\Bc$ considered in \cite{Kor26}.
The logarithm of the two-valued formal group obtained from
\eqref{eq:bcN-exponential} by the modulus square construction is
\begin{equation}\label{eq:bcN-two-valued-logarithm}
 B_{\Bc,N}(x)
 =\left(
   \int_0^{\sqrt{x}}
   \frac{\dd t}{\sqrt{1+a_1t^2+\cdots+a_Nt^{2N}}}
  \right)^2.
\end{equation}
If $N\geq n$, the coefficient
\[
 \beta_n=[x^{n+1}]B_{\Bc,N}(x)=\Bc_N(b_n)
\]
depends only on $a_1,\ldots,a_n$. Put
\begin{equation}\label{eq:def-delta-n}
 \delta_n
 =\frac{d_n}{2^{\nu_2(d_n)}}
 =\frac{n+1}{2^{\nu_2(n+1)}}
  \lcm(1,3,5,\ldots,2n+1).
\end{equation}
Thus $\delta_n$ is the odd part of $d_n$.

\begin{theorem}[(The odd denominator theorem)]\label{thm:bc-odd-denominator}
For every $n\geq1$ and every $N\geq n$, the polynomial
\[
 \delta_n\beta_n\in\mathbb Z[a_1,\ldots,a_n]
\]
is primitive. Consequently, $\delta_n$ is the exact common denominator of
$\Bc_N(b_n)$. In particular, the genus $\Bc_n$ detects exactly the odd part
of $d_n$ from \eqref{eq:def-dn}.
\end{theorem}

The proof is given in \cref{sec:odd-denominators}.

The first coefficients are
\[
 \beta_1=-\frac{a_1}{3},\qquad
 \beta_2=\frac{8a_1^2-9a_2}{45},\qquad
 \beta_3=-\frac{12a_1^3-26a_1a_2+15a_3}{105}.
\]

Recall from \eqref{eq:ring-hierarchy-main} that $\widetilde\Lambda$ is generated
by the coefficients of $\Theta_1/2$ and $\Theta_2$, and that
$\St=\widetilde\Lambda\cap\OmU$.

\begin{theorem}[(Integrality threshold on the Stong ring)]
\label{thm:bc-stong-threshold}
The first three odd genera have integral restrictions
\[
 \Bc_1(\St)\subset\mathbb Z[2a_1],\qquad
 \Bc_2(\St)\subset\mathbb Z[2a_1,a_2],\qquad
 \Bc_3(\St)\subset\mathbb Z[2a_1,a_2,2a_3].
\]
For every $N>3$ one has
\[
 \Bc_N(\St)\not\subset\mathbb Z[a_1,\ldots,a_N].
\]
\end{theorem}

The proof is given in \cref{sec:integral-truncations}.

\subsection{Chern number divisibility}

The primitive classes $C_n$ nevertheless have highly divisible characteristic
numbers. The next theorem states that the common divisor of their Chern
numbers is exactly the denominator of $b_n$ from \eqref{b_n_as_fraction}.

We use $m_\lambda[X]$ for monomial symmetric tangent Chern numbers and $c_I[X]$
for ordinary Chern numbers throughout. Both the $m_\lambda$ and the products
$c_I$ form integral bases of symmetric functions, so their characteristic
numbers have the same common divisors.

\begin{theorem}\label{thm:chern-divisibility}
For every $n\geq1$, the class $C_n=d_nb_n$ belongs to $\Lambda^{-4n}\subset\LambdaSt^{-4n}$. Moreover,
for every partition $\lambda$ of $2n$,
\[
 m_\lambda[C_n]\in d_n\mathbb Z,
\]
and
\[
 c_{2n}(C_n)=d_n.
\]
Consequently,
\[
 \gcd\{m_\lambda[C_n]\mid\lambda\vdash2n\}=d_n.
\]
\end{theorem}

The proof is given in \cref{sec:chern-divisibility}.

\subsection{Oddification of Hirzebruch genera}

On a stably quaternionic manifold the formal Chern roots occur in opposite pairs. This
makes the even part of a characteristic series sufficient for computations on
the Stong ring and leads to a remarkable oddification principle. Recall that $\Bc := \Bc_3$.

Let $j\colon\OmSpfree^{\ast}\to\OmU$ be the forgetful map. The Stong ring is
the $2$-primary saturation of its image. Equivalently, Buchstaber's result
\cite[Theorem~4.1 and Lemma~6.6]{Buc78} gives
\[
 \St=j(\OmSpfree^{\ast})\!\left[\frac12\right]\cap\OmU.
\]
Let $R$ be a commutative $\mathbb Q$-algebra,
and let $Q(z)\in R[[z]]$ satisfy $Q(0)=1$. Denote by
$\varphi_Q\colon\OmU\to R$ the corresponding Hirzebruch genus. Its \emph{oddification} is the genus whose characteristic
series is
\begin{equation}\label{eq:oddification-Q}
 \widetilde Q(z)=\sqrt{Q(z)Q(-z)},
 \qquad \widetilde Q(0)=1.
\end{equation}

\begin{theorem}[(Oddification theorem)]
\label{thm:oddification-stong}
The genera determined by $Q$ and $\widetilde Q$ have the same restrictions to
the Stong ring:
\[
 \left.\varphi_Q\right|_{\St}
 =\left.\varphi_{\widetilde Q}\right|_{\St}.
\]
\end{theorem}

The proof is given in \cref{sec:oddification}.

For a triple $(\alpha,\beta,\gamma)$, write
\[
 \Bc^{(\alpha,\beta,\gamma)}(X)
 =\left.\Bc(X)\right|_{a_1=\alpha,\ a_2=\beta,\ a_3=\gamma}.
\]

\begin{corollary}\label[corollary]{cor:bc-euler-todd-stong}
Let $\chi$ denote the top Chern number Hirzebruch genus, so that
$\chi(X)=c_{2n}[X]$ for $X\in\St^{-4n}$. Then
\[
 \left.\Bc^{(-3,3,-1)}\right|_{\St}
 =\left.\chi\right|_{\St},
 \qquad
 \left.\Bc^{(-1/4,0,0)}\right|_{\St}
 =\left.\Td\right|_{\St}.
\]
\end{corollary}

The proof is given in \cref{sec:oddification}.

Put
\[
 \Och=\left.\Bc\right|_{a_3=0}.
\]
This is the Ochanine genus in the normalization used here.

\begin{theorem}[(The $\Bc$ genus is not determined by $\Och$ and $\chi$)]
\label{thm:bc-not-determined}
For every $n\geq5$, there exist classes
$X_n,Y_n\in\Lambda^{-4n}\subset\St^{-4n}$ such that $X_n$, $Y_n$, and
$X_n-Y_n$ are all undecomposable in $\OmU$, even after rationalization, and
\[
 \Och(X_n)=\Och(Y_n)=0,
 \qquad
 \chi(X_n)=c_{2n}[X_n]=c_{2n}[Y_n]=\chi(Y_n).
\]
Here $\chi$ is the top Chern number Hirzebruch genus. However,
\[
 \Bc(X_n)-\Bc(Y_n)
 =128(4a_1)^{n-5}a_3(a_1^2-3a_2)\ne0.
\]
\end{theorem}

The proof is given in \cref{sec:bc-not-determined}.

\subsection{Spin integrality below and string integrality through dimension 24}
\label{sec:spin-main}

The Buchstaber genus also defines a genus on oriented cobordism. Using the
forgetful isomorphism after inverting $2$ \cite{Sto67}, we put
\[
 \Omega_*^{\mathrm{SO}}\longrightarrow
 \Omega_*^{\mathrm{SO}}\!\left[\frac12\right]
 \cong\Omega_*^{\Sp}\!\left[\frac12\right]
 \xrightarrow{\ \Bc\circ j\ }
 \mathbb Z\!\left[\frac12,a_1,a_2,a_3\right],
\]
where $j$ forgets the quaternionic structure. The target follows from
\cref{thm:bc-stong-threshold}. We denote this oriented genus by $\Bc$ as well
and call a value \emph{integral} if it belongs to
$\mathbb Z[a_1,a_2,a_3]$, without inverting $2$. Since the characteristic
series of $\Bc$ is even, this extension agrees with the original genus on
stably complex manifolds.

\begin{theorem}[(Spin integrality below dimension 24)]
\label{thm:bc-spin-below24}
For every closed smooth spin manifold $M$ with $\dim_{\mathbb R}M<24$,
\[
 \Bc(M)\in\mathbb Z[a_1,a_2,a_3].
\]
In particular, the $\Bc$ genus is integral on every compact
Calabi--Yau manifold\footnote{By Calabi--Yau manifolds we mean stably complex manifolds $M$ with $c_1(\tau_M) = 0$ in $H^2(M;\mathbb Z)$.} of real dimension less
than $24$.
\end{theorem}

\begin{theorem}[(String integrality through dimension 24)]
\label{thm:bc-string-through24}
For every closed smooth string manifold $M$ with $\dim_{\mathbb R}M\leq24$
$($in particular, $M$ is spin and $p_1(\tau_M)/2=0)$, one has
\[
\Bc(M)\in\mathbb Z[a_1,a_2,a_3].
\]
\end{theorem}

The proofs are given in \cref{sec:spin-integrality}.

The dimension bound in \cref{thm:bc-spin-below24} is sharp for spin
manifolds. The Anderson--Brown--Peterson computation
\cite{ABP67} gives a closed spin $24$-manifold $\Sigma_{\mathrm{ABP}}$ for
which $\langle s_6(p),[\Sigma_{\mathrm{ABP}}]\rangle$ is odd, where $s_6(p)$
is the sixth power sum of the Pontryagin roots. For this manifold we obtain
\[
 \Bc(\Sigma_{\mathrm{ABP}})\notin\mathbb Z[a_1,a_2,a_3].
\]
Indeed, the coefficient of $a_3^2$ in this genus value has exact denominator
$8$. See \cref{ex:abp-spin24}.

\section{Proof of the exact denominator theorem}\label{sec:main-proof}

The exact denominator argument in \cref{thm:main} has four parts. We first
establish an integral upper bound, then recall the integral index genus and
compute its specialized two-valued logarithm, isolate the necessary
arithmetic in a $p$-adic lemma, and use the genus to prove the matching lower
bound. We conclude by computing the Milnor number of $C_n$ to prove
undecomposability and irreducibility.

\subsection{The integral upper bound}\label{sec:upper}

The upper bound comes from a residue expression for $b_n$. Mishchenko's
formula controls the ordinary denominators, while a parity lemma supplies the
single extra factor needed at the prime $2$.

Put
\begin{equation}\label{eq:def-q}
 q(u)=-\frac{\bar u}{u}\in 1+u\OmU[[u]],
 \qquad
 X=-u\bar u=u^2q(u).
\end{equation}
By the Mishchenko formula,
\begin{equation}\label{eq:mishchenko}
 g(u)=\sum_{j\geq0}\frac{[\CP^j]}{j+1}u^{j+1},
 \qquad
 g'(u)=\sum_{j\geq0}[\CP^j]u^j\in\OmU[[u]].
\end{equation}

We first record the residue calculation used below.

\begin{lemma}[(Ramified residue formula)]\label[lemma]{lem:ramified-residue}
For every
\[
H(X)\in(\OmU\otimes\mathbb{Q})[[X]]
\]
and every \(m\geq0\),
\[
[X^m]H(X)
=
\frac12\res_{u=0}
H\bigl(X(u)\bigr)X(u)^{-m-1}X'(u)\,\dd u.
\]
\end{lemma}

\begin{proof}
It is enough to take \(H(X)=X^r\). If \(r\neq m\), then
\[
\res_{u=0}X(u)^{r-m-1}X'(u)\,\dd u
=
\frac{1}{r-m}
\res_{u=0}\frac{\dd}{\dd u}\bigl(X(u)^{r-m}\bigr)\,\dd u
=0.
\]
If \(r=m\), then \(X(u)=u^2q(u)\) and \(q(0)=1\), so
\[
\res_{u=0}\frac{X'(u)}{X(u)}\,\dd u
=
\ord_uX
=
2.
\]
The result follows.
\end{proof}

The only $2$-primary point needed in the upper bound is contained in the next
lemma.

\begin{lemma}[(The $2$-divisibility lemma)]\label[lemma]{lem:parity}
For $r\geq1$, one has
\begin{equation}\label{eq:zero-coefficient}
 [u^{2r-1}]\,g'(u)q(u)^{-r}=0.
\end{equation}
Moreover, if $M$ is a power of $2$ and $0<r<M$, then
\begin{equation}\label{eq:even-coefficient}
 [u^{2r-1}]\,g'(u)q(u)^{-(M+r)}\in2\OmU.
\end{equation}
\end{lemma}

\begin{proof}
For the first identity, write the coefficient as a formal residue:
\[
 [u^{2r-1}]g'(u)q(u)^{-r}
 =\res_{u=0}g'(u)X^{-r}\,\dd u.
\]
Make the strict change of variable $t=g(u)$.  Since $u=f(t)$, we have
\[
 X=-f(t)f(-t).
\]
The series on the right is an even series in $t$.  Hence $X^{-r}$ contains
only even powers of $t$, and therefore
\[
 \res_{t=0}X^{-r}\,\dd t=0.
\]
This proves \eqref{eq:zero-coefficient}.

We prove \eqref{eq:even-coefficient} by induction on $M$.  Reduce the
coefficients modulo $2$ and put
\[
 S_j(u)=g'(u)q(u)^{-j}.
\]
Since $M$ is a power of $2$, the Frobenius identity gives
\begin{equation}\label{eq:frobenius-q}
 q(u)^{-M}=\bigl(q(u)^{-1}\bigr)^M
 \in (\OmU/2\OmU)[[u^M]].
\end{equation}
If $0<r\leq M/2$, then $2r-1<M$, and \eqref{eq:frobenius-q} gives
\[
[u^{2r-1}]S_{M+r}(u)
\equiv
[u^{2r-1}]S_r(u)
=0
\pmod 2
\]
by \eqref{eq:zero-coefficient}.

Now suppose that $M/2<r<M$ and write $r=M/2+s$, where $0<s<M/2$.
Since $2r-1=M+2s-1<2M$, only the terms of degrees $0$ and $M$ in
$q(u)^{-M}$ can contribute.  Thus, for some $\alpha\in\OmU/2\OmU$,
\[
 [u^{2r-1}]S_{M+r}(u)
 \equiv [u^{2r-1}]S_r(u)
   +\alpha [u^{2s-1}]S_r(u)
 \pmod 2.
\]
The first term is zero by \eqref{eq:zero-coefficient}. For the second term, apply the induction hypothesis to the smaller power \(M/2\) and the integer \(s\). Since \(0<s<M/2\),
\[
[u^{2s-1}]S_{M/2+s}\equiv0\pmod 2.
\]
Since \(M/2+s=r\), this is precisely
\[
[u^{2s-1}]S_r\equiv0\pmod 2.
\]
Thus the required coefficient vanishes modulo \(2\). This completes the induction.
\end{proof}

We now obtain the coefficient formula from which the upper bound follows.

\begin{proposition}[(A residue formula for $b_n$)]\label[proposition]{prop:residue-formula-bn}
Let $m\geq2$.  Then
\begin{equation}\label{eq:bn-residue}
 b_{m-1}=\frac{(-1)^{m-1}}{m}
 [u^{2m-1}]\,g(u)g'(u)q(u)^{-m}.
\end{equation}
\end{proposition}

\begin{proof}
Set
\[
 C(X)=-B(-X).
\]
By \eqref{eq:def-B} and \eqref{eq:def-q},
\[
 C(X)=g(u)^2,
 \qquad X=u^2q(u).
\]
Moreover,
\[
 [X^m]C(X)=(-1)^{m-1}b_{m-1}.
\]
By \cref{lem:ramified-residue},
\begin{align*}
 [X^m]C(X)
 &=\frac12\res_{u=0}g(u)^2X^{-m-1}X'(u)\,\dd u\\
 &=\frac1m\res_{u=0}g(u)g'(u)X^{-m}\,\dd u.
\end{align*}
The second equality follows by integration by parts.  Since
$X^{-m}=u^{-2m}q(u)^{-m}$, the last residue is the coefficient on the
right-hand side of \eqref{eq:bn-residue}.
\end{proof}

\begin{lemma}[(The coefficient lemma)]\label[lemma]{lem:upper-coefficient}
For every \(m\geq1\),
\begin{equation}\label{eq:coefficient-integrality}
\frac{d_m}{m+1}
[u^{2m+1}]\,g(u)g'(u)q(u)^{-(m+1)}
\in\OmU.
\end{equation}
\end{lemma}

\begin{proof}
By \eqref{eq:mishchenko},
\begin{equation}\label{eq:expanded-residue}
 [u^{2m+1}]g(u)g'(u)q(u)^{-(m+1)}
 =\sum_{j=0}^{2m}
 \frac{[\CP^j]}{j+1}
 [u^{2m-j}]g'(u)q(u)^{-(m+1)}.
\end{equation}
Suppose first that \(m+1\) is a power of \(2\). Then
\[
\frac{d_m}{m+1}=L_{2m+1},
\] where $L_{2m+1}=\lcm(1,2,\ldots,2m+1)$.
Since \(1\leq j+1\leq2m+1\), multiplication of
\eqref{eq:expanded-residue} by \(L_{2m+1}\) makes every summand
integral.

Now suppose that \(m+1\) is not a power of \(2\). Let \(N\) be the largest
power of \(2\) not exceeding \(2m+1\). Then
\[
 m+1<N<2m+2.
\]
In this case $d_m/(m+1)=L_{2m+1}/2$.
For \(j+1\neq N\), the integer
\[
 \frac{L_{2m+1}}{2(j+1)}
\]
is an integer. It remains to consider the summand with \(j+1=N\). Put
\[
 M=\frac N2,
 \qquad
 r=m+1-M.
\]
Then \(M\) is a power of \(2\), \(0<r<M\), and
\[
 2m-N+1=2r-1.
\]
By \cref{lem:parity},
\[
 [u^{2m-N+1}]g'(u)q(u)^{-(m+1)}
 =[u^{2r-1}]g'(u)q(u)^{-(M+r)}
 \in2\OmU.
\]
This supplies the remaining factor \(2\) and proves
\eqref{eq:coefficient-integrality}.
\end{proof}

\begin{corollary}[(Upper bound)]\label[corollary]{cor:upper}
For every $n\geq1$,
\[
 d_nb_n\in\OmU,
\]
where $d_n$ is given by \eqref{eq:def-dn}.
\end{corollary}

\begin{proof}
Combine \eqref{eq:bn-residue} with \(m=n+1\) and
\eqref{eq:coefficient-integrality}.
\end{proof}

\subsection{The integral index genus \texorpdfstring{$\Phi_t$}{Phi(t)}}\label{sec:genus}

To prove that the upper bound is exact, we use the index genus $\Phi_t$
from \eqref{eq:intro-Phi-roots}. Its exponential is
\begin{equation}\label{eq:genus-exponential}
 f_{\Phi,t}(z)=\frac{z}{Q_{\Phi,t}(z)}
 =\frac{1-e^{-z}}{1+t(e^z-1)}.
\end{equation}
The quotient is understood as a formal power series in $z$ with
coefficients in $\mathbb Q[t]$.

\begin{proposition}[(Integrality of $\Phi_t$)]\label[proposition]{prop:integral-genus}
The index genus defines a ring homomorphism
\[
 \Phi_t\colon\OmU\longrightarrow\mathbb Z[t].
\]
In particular, $\Phi_j(M)\in\mathbb Z$ for every stably complex manifold
$M$ and every $j$.
\end{proposition}

\begin{proof}
The coefficients $\Phi_j(M)$ are the Hattori--Stong $\mathrm{K}$-theory indices
recalled in \eqref{eq:intro-hattori-stong-phi}, with a single
$\gamma$-operation, and are therefore integers \cite{Hat66,Sto65}.
The formula for the characteristic series in \eqref{eq:intro-Phi-roots} gives
multiplicativity and hence the asserted ring homomorphism.
\end{proof}

Let $B_{\Phi,t}(x)$ be obtained by applying $\Phi_t$ coefficientwise to the
universal two-valued logarithm $B(x)$. Thus
\[
 B_{\Phi,t}\bigl(f_{\Phi,t}(z)f_{\Phi,t}(-z)\bigr)=-z^2.
\]

\begin{lemma}[(The specialized logarithm)]\label[lemma]{lem:specialized-B}
One has
\begin{equation}\label{eq:B-Phi}
 B_{\Phi,t}(x)=4
 \arcsin^2\!\left(\frac{\sqrt{x}}{2\sqrt{1+cx}}\right),
\end{equation} where $c=t(1-t)$.
Consequently, for $n\geq1$,
\begin{equation}\label{eq:coefficient-B-Phi}
 [x^{n+1}]B_{\Phi,t}(x)
 =\sum_{k=1}^{n+1}(-1)^{n+1-k}
 \frac{\binom{n}{k-1}}
 {k^2\binom{2k-1}{k-1}}
 c^{n+1-k}.
\end{equation}
\end{lemma}

\begin{proof}
A direct simplification of
$f_{\Phi,t}(z)f_{\Phi,t}(-z)$ gives
\[
 f_{\Phi,t}(z)f_{\Phi,t}(-z)
 =-\frac{4\sinh^2(z/2)}{1+4c\sinh^2(z/2)}.
\]
Writing \(x=f_{\Phi,t}(z)f_{\Phi,t}(-z)\), we obtain
\[
\frac{x}{4(1+cx)}=-\sinh^2\!\left(\frac{z}{2}\right).
\]
The formal identity \(\arcsin(i\sinh w)=iw\) now gives
\eqref{eq:B-Phi}.

Now use
\[
 (\arcsin z)^2
 =\frac12\sum_{k\geq1}
 \frac{(2z)^{2k}}{k^2\binom{2k}{k}}.
\]
Substitution in \eqref{eq:B-Phi}, together with
\[
 \binom{2k}{k}=2\binom{2k-1}{k-1}
\]
and
\[
 (1+cx)^{-k}
 =\sum_{j\geq0}(-1)^j\binom{k+j-1}{j}c^jx^j,
\]
gives \eqref{eq:coefficient-B-Phi}.
\end{proof}

\subsection{Arithmetic of the specialized coefficients}\label{sec:arithmetic}

To obtain the lower bound, it remains to prove that
$d_n[x^{n+1}]B_{\Phi,t}(x)$ is primitive in $\mathbb Z[t]$. This reduces to
the arithmetic lemma proved in this subsection.

For $m\geq2$ and $1\leq k\leq m$, put
\begin{equation}\label{eq:r-mk}
 r_{m,k}=
 \frac{\binom{m-1}{k-1}}
 {k^2\binom{2k-1}{k-1}}.
\end{equation}

\begin{lemma}[(Arithmetic lemma)]\label[lemma]{lem:arithmetic}
For \(m\ge2\), put
\[
\alpha_{m,k}=d_{m-1}r_{m,k},
\qquad 1\le k\le m.
\]
Then $\alpha_{m,k}\in\mathbb Z$ for all \(k\), and
\[
\gcd(\alpha_{m,1},\ldots,\alpha_{m,m})=1.
\]
Thus the numbers \(r_{m,k}\) have the common denominator \(d_{m-1}\)
with relatively prime numerators.
\end{lemma}

\begin{proof}
Put \(D=d_{m-1}\), so that
\[
\alpha_{m,k}=Dr_{m,k}.
\]
We shall prove the stronger identity
\begin{equation}\label{eq:max-valuation}
\max_{1\leq k\leq m}\bigl(-\vp(r_{m,k})\bigr)=\vp(D)
\end{equation}
for every prime \(p\).

The identity
\begin{equation}\label{eq:r-simplified}
r_{m,k}
=
\frac{1}{m}
\frac{\binom{m}{k}}
{k\binom{2k-1}{k-1}}
\end{equation}
gives
\begin{equation}\label{eq:valuation-r}
-\vp(r_{m,k})
=
\vp(m)
+
\vp\!\left(k\binom{2k-1}{k-1}\right)
-
\vp\binom{m}{k}.
\end{equation}

First let \(p\) be odd and put
\[
e=\left\lfloor\log_p(2m-1)\right\rfloor.
\]
Because $p$ is odd and $2m$ cannot be a power of $p$,
\begin{equation}\label{eq:odd-D}
\vp(D)=\vp(m)+e.
\end{equation}
By Legendre's formula,
\begin{align}
\vp\!\left(k\binom{2k-1}{k-1}\right)
&=
\vp\!\left(\frac{(2k-1)!}{(k-1)!^2}\right)\notag\\
&=
\sum_{j\geq1}
\left(
\left\lfloor\frac{2k-1}{p^j}\right\rfloor
-
2\left\lfloor\frac{k-1}{p^j}\right\rfloor
\right).
\label{eq:legendre-sum}
\end{align}
If \(k-1=Ap^j+s\), where \(0\leq s<p^j\), then the \(j\)-th
summand equals
\[
\left\lfloor\frac{2s+1}{p^j}\right\rfloor.
\]
It is either \(0\) or \(1\), and it vanishes when \(p^j>2k-1\).
Consequently,
\[
\vp\!\left(k\binom{2k-1}{k-1}\right)\leq e,
\]
and \eqref{eq:valuation-r} gives
\begin{equation}\label{ineq}
-\vp(r_{m,k})\leq\vp(m)+e=\vp(D).
\end{equation}

We now show that equality in \eqref{ineq} is attained. Put \(t=p^e\).
If \(t\leq m\), take \(k=t\). Lucas' theorem gives
\[
\vp\binom{m}{k}=0,
\]
while \(k-1=p^e-1\), so every summand in
\eqref{eq:legendre-sum}, for \(1\leq j\leq e\), equals \(1\).
Hence equality holds in \eqref{ineq}.

Suppose that \(m<t\). Then
\[
\frac{t+1}{2}\leq m<t.
\]
Put \(h=(p-1)/2\) and \(q=(t+1)/2\). In base \(p\), the digits of
\(q\) are
\[
q_0=h+1,
\qquad
q_i=h\quad (1\leq i<e).
\]
Write \(m=\sum_{i=0}^{e-1}m_ip^i\). If \(m=q\), set \(\ell=0\).
Otherwise, let \(\ell\) be the largest index for which \(m_\ell\neq q_\ell\).
In the latter case, $m>q$ and hence $m_\ell>q_\ell$. In the former case,
$m_0=q_0=h+1$. Define \(k\) by the base-\(p\)
digits
\[
k_i=
\begin{cases}
0,&i<\ell,\\
h+1,&i=\ell,\\
h,&i>\ell.
\end{cases}
\]
Every digit of \(k\) is at most the corresponding digit of \(m\), so
Lucas' theorem gives
\[
\vp\binom{m}{k}=0.
\]
The digits of \(k-1\) are equal to \(p-1\) below position \(\ell\) and
to \(h\) from position \(\ell\) onwards. Therefore, for every
\(1\leq j\leq e\),
\[
(k-1)\bmod p^j\geq\frac{p^j-1}{2},
\]
and hence every summand in \eqref{eq:legendre-sum} equals \(1\).
Thus
\[
\vp\!\left(k\binom{2k-1}{k-1}\right)=e,
\]
so equality in \eqref{ineq} is again attained.

It remains to consider \(p=2\). Legendre's formula gives
\begin{equation}\label{eq:binary-weight}
\nu_2\!\left(k\binom{2k-1}{k-1}\right)=s_2(k-1),
\end{equation}
where \(s_2(r)\) denotes the number of ones in the binary expansion of
\(r\).

If \(m=2^e\), then \(s_2(k-1)\leq e\) for every \(k\leq m\), and
\eqref{eq:valuation-r} gives
\[
-\nu_2(r_{m,k})\leq\nu_2(m)+e=\nu_2(D).
\]
For \(k=m\),
\[
s_2(k-1)=e,
\qquad
\binom{m}{k}=1,
\]
so equality is attained.

Suppose that \(m\) is not a power of \(2\), and put
\[
e=\left\lceil\log_2m\right\rceil.
\]
Then \(s_2(k-1)\leq e-1\) for every \(k\leq m\), and therefore
\[
-\nu_2(r_{m,k})
\leq
\nu_2(m)+e-1
=
\nu_2(D).
\]
Taking \(k=2^{e-1}\), Lucas' theorem gives
\[
\nu_2\binom{m}{k}=0,
\qquad
s_2(k-1)=e-1,
\]
so equality is attained in this case as well. This proves
\eqref{eq:max-valuation} for every prime \(p\).

It follows from \eqref{eq:max-valuation} that
\[
\vp(\alpha_{m,k})
=
\vp(D)+\vp(r_{m,k})
\geq0
\]
for every prime \(p\) and every \(k\). Hence
\[
\alpha_{m,k}\in\mathbb Z.
\]

Finally, for each prime \(p\), choose an index \(k\) for which the maximum
in \eqref{eq:max-valuation} is attained. Then
\[
\vp(\alpha_{m,k})
=
\vp(D)+\vp(r_{m,k})
=
0.
\]
Thus no prime divides all the integers \(\alpha_{m,1},\ldots,\alpha_{m,m}\).
Consequently,
\[
\gcd(\alpha_{m,1},\ldots,\alpha_{m,m})=1.
\]
\end{proof}

\begin{corollary}\label[corollary]{cor:primitive-polynomial}
For $n\geq1$ and an indeterminate $c$, the polynomial
\begin{equation}\label{eq:primitive-index-polynomial}
 \mathcal R_n(c):=d_n\sum_{k=1}^{n+1}(-1)^{n+1-k}
 r_{n+1,k}c^{n+1-k}\in\mathbb Z[c]
\end{equation}
is primitive.
\end{corollary}

\begin{proof}
Its coefficients, up to signs, are the integers $\alpha_{n+1,k}$ of
\cref{lem:arithmetic}, and hence have greatest common divisor $1$.
\end{proof}

To pass from $\mathcal R_n(c)$ to the genus value, we substitute $c=t(1-t)$.
The following observation shows that this preserves primitivity.

\begin{lemma}[(Injectivity of the substitution)]\label[lemma]{lem:substitution}
For every prime $p$, the ring homomorphism
\[
 (\mathbb Z/p)[c]\longrightarrow(\mathbb Z/p)[t],
 \qquad c\longmapsto t(1-t),
\]
is injective.
\end{lemma}

\begin{proof}
A nonzero polynomial of degree $r$ with leading coefficient $\alpha$ maps
to a polynomial of degree $2r$ with leading coefficient $(-1)^r\alpha$,
which is also nonzero when $p=2$.
\end{proof}

\subsection{The lower bound from the index genus}\label{sec:lower}

We now apply $\Phi_t$ to the universal coefficient $C_n$. The primitivity
of the specialized polynomial rules out every proper divisor of $d_n$ and
therefore completes the exact denominator argument.

\begin{corollary}[(Lower bound)]\label[corollary]{cor:lower}
Let
\[
 C_n=d_nb_n\in\OmU
\]
be the class supplied by \cref{cor:upper}.  Then $C_n$ is primitive in
$\OmU$.
\end{corollary}

\begin{proof}
Apply the integral index genus $\Phi_t$ of \cref{prop:integral-genus}.
By \eqref{eq:coefficient-B-Phi} and \eqref{eq:primitive-index-polynomial},
\[
 \Phi_t(C_n)=d_n[x^{n+1}]B_{\Phi,t}(x)=\mathcal R_n\bigl(t(1-t)\bigr).
\]
The polynomial $\mathcal R_n(c)$ is primitive by \cref{cor:primitive-polynomial}.
By \cref{lem:substitution}, its reduction modulo any prime remains nonzero
after the substitution $c=t(1-t)$. Thus $\Phi_t(C_n)$ is primitive in
$\mathbb Z[t]$.

If $C_n=rD$ in $\OmU$ for an integer $r>1$, then
\[
 \Phi_t(C_n)=r\Phi_t(D).
\]
Since $\Phi_t(D)\in\mathbb Z[t]$, every coefficient of $\Phi_t(C_n)$ would
be divisible by $r$, a contradiction. Hence $C_n$ is primitive.
\end{proof}

\begin{proof}[Proof of \cref{thm:main}]
The inclusion $C_n=d_nb_n\in\OmU$ follows from \cref{cor:upper}, and the
primitivity of $C_n$ follows from \cref{cor:lower}. If $Nb_n\in\OmU$ for a
positive integer $N$, then
$(N/d_n)C_n$ is integral. Primitivity of $C_n$ in the free abelian group
$\Omega_{\Uu}^{-4n}$ implies $d_n\mid N$. Thus $d_n$ is the exact denominator
of $b_n$.

To prove undecomposability, fix $n\geq1$ and take the characteristic series
\[
 Q(z)=1+z^{2n}.
\]
For its Hirzebruch genus $\varphi_Q$ and a stably complex manifold $M$ of
complex dimension $2n$, one has
\[
 \varphi_Q(M)
 =\left\langle\sum_i x_i^{2n},[M]\right\rangle
 =s_{2n}(M),
\]
where the $x_i$ are the tangent Chern roots. This identity also holds for
rational cobordism classes of the same degree. The exponential is
$f_Q(z)=z/Q(z)$. By \eqref{eq:def-exp-B}, the specialized two-valued
logarithm $B_Q$ satisfies
\[
 B_Q^{-1}(-z^2)=-\frac{z^2}{(1+z^{2n})^2},
 \qquad
 B_Q^{-1}(x)=\frac{x}{(1+(-1)^n x^n)^2}.
\]
Compositional inversion to order $n+1$ gives
\[
 B_Q(x)=x+2(-1)^n x^{n+1}+O(x^{n+2}).
\]
Consequently,
\[
 s_{2n}(C_n)=\varphi_Q(C_n)
 =d_n[x^{n+1}]B_Q(x)=2(-1)^n d_n,
\]
which proves the asserted Milnor number formula.

The power sum is additive under Whitney sums. Hence, by dimensional
reasons, $s_{2n}(M\times N)=0$ whenever $M$ and $N$ have positive complex
dimensions summing to $2n$. Thus $s_{2n}$ vanishes on
$(\mathfrak a^2)^{-4n}$ and on its rationalization. Its nonzero value on
$C_n$ proves undecomposability, even over $\mathbb Q$.

Finally, $\OmU$ is a graded polynomial ring over $\mathbb Z$, so any
factorization of the homogeneous class $C_n$ has homogeneous factors. Two
factors of negative degree would make $C_n$ decomposable. A factor of
degree zero is an integer and must be a unit by primitivity. Therefore
$C_n$ is irreducible.
\end{proof}

\section{Logarithmic numerators in the Stong and coefficient rings}
\label{sec:stong-numerators}

This section proves
\cref{thm:stong-numerators,thm:sp-denominator,thm:Cn-Lambda}. We first recall
the topological origin of the invariant differentiation and generalized shift,
then derive a closed formula for the series $H_q$. We next extract from these
series an odd integer-valued polynomial whose derivative at zero is a
coefficient of the logarithm $B(x)$. The exact arithmetic factor needed to
recover $C_n$ in $\LambdaSt$ is provided by a general lemma on odd
integer-valued polynomials. We then lift the quotient Stong manifold formula
to genuine quaternionic cobordism and prove that $\widetilde b_n$ has exact
denominator $2d_n$ and primitive numerator $\widehat C_n$.
Finally, the restriction of $\Theta_1$ to the diagonal gives the additional
$2$-primary divisibility needed to place $C_n$ in the smaller ring $\Lambda$.

\subsection{Topological origin of the generalized shift}

Let $A^{\U}$ be the algebra of stable cohomology operations in complex
cobordism. For a finite CW complex $X$, Buchstaber introduced the ring
\cite[\S~4, (26)]{Buc78}
\begin{equation}\label{eq:functor-L-definition}
 \mathcal L^{*}(X)
 =\operatorname{Hom}_{A^{\U}}
   \bigl(\U^{*}(\MSp),\U^{*}(X)\bigr)
 \subset\U^{*}(X).
\end{equation}
For locally finite complexes this definition is extended by inverse limits over
the finite skeleta. The coefficient ring of this functor is the Stong ring
\cite[Lemma~6.6]{Buc78},
\begin{equation}\label{eq:L-point-Stong}
 \mathcal L^{*}(\mathrm{pt})=\St,
\end{equation}
and, if the homology of $X$ has no $2$-torsion, the subring
\[
 \mathcal L^{*}(X\times\CP^\infty)
 \subset\U^{*}(X)[[u]]
\]
is closed under the invariant differentiation
\[
 D=\frac{\dd}{\dd g(u)}
   =\frac{1}{g'(u)}\frac{\dd}{\dd u}.
\]
Here $u=c_1^{\U}(\eta)$ is the universal complex cobordism Chern class
\cite[Lemma~4.5]{Buc78}.

The two-valued formal group describes the corresponding extension in the
Buchstaber--Novikov coordinate
$x=u\bar u=\iota^*c_2^{\U}(\zeta_{\mathbb C})$.
Writing $x_i=u_i\bar u_i$ and
$\mathcal L^{4*}=\bigoplus_{r}\mathcal L^{4r}$, Buchstaber proved
\cite[Theorem~5.2(3)]{Buc78}
\begin{equation}\label{eq:L-quadratic-extension}
 \mathcal L^{4*}(\CP^\infty\times\CP^\infty)
 \cong
 \St[[x_1,x_2,Z]]
 \big/
 \bigl(Z^2-\Theta_1(x_1,x_2)Z+\Theta_2(x_1,x_2)\bigr).
\end{equation}
Moreover, the stable map constructed in \cite[Theorem~2.11]{Buc78}
induces $4\Delta$, where $\Delta$ is the averaging operator defined below.
Thus these operators have a topological origin in cobordism and stable
transfer. The normalization used below is fixed by the quotient
Stong manifold expansion.

\subsection{Generalized shifts and the series \texorpdfstring{$H_q$}{Hq}}

Put
\[
 E(x)=B^{-1}(x),
\]
and denote the two values of the universal two-valued formal group by
\begin{equation}\label{eq:zpm-B-coordinate}
 z_{\pm}(x,y)
 =E\!\left(
   \bigl(\sqrt{B(x)}\mathbin{\pm}\sqrt{B(y)}\bigr)^2
  \right).
\end{equation}
Thus
\[
 z_++z_-=\Theta_1(x,y),
 \qquad
 z_+z_-=\Theta_2(x,y).
\]
For a series $\varphi(x)\in(\OmU\otimes\Q)[[x]]$, let
\[
 \Delta\varphi(x,y)
 =\frac12\bigl(\varphi(z_+(x,y))+\varphi(z_-(x,y))\bigr).
\]
The first canonical invariant differential operator is
\[
 \mathcal D_1\varphi(x)
 =\left.
   \frac{\partial}{\partial y}\Delta\varphi(x,y)
  \right|_{y=0}.
\]
For $x=u\bar u$, it is related to the invariant derivative $D$ by
$D^2\varphi(x)=-2\mathcal D_1\varphi(x)$ \cite[(50)]{Buc78}.
For $q\geq1$, let $\Psi_q(x)$ be the unique solution of
\[
 \mathcal D_1\Psi_q(x)=x^q,
 \qquad
 \Psi_q(0)=0,
\]
and put
\begin{equation}\label{eq:def-uq-Hq}
 u_q(x,y)=\Delta\Psi_q(x,y),
 \qquad
 H_q(x,y)=\frac{\partial^2u_q}{\partial x\,\partial y}(x,y).
\end{equation}

The normalization of the formula for $H_q$ is important. Buchstaber's
coefficient formula contains the factor $2$ below. Nadiradze later realized
this divisibility geometrically by a free involution and its quotient.

\begin{proposition}
\label[proposition]{prop:Hq-stong-expansion}
For every $q\geq1$, one has
\begin{equation}\label{eq:Hq-stong-double}
 2H_q(x,y)
 =-\sum_{k,l\geq0}[M(k,l;q)]x^ky^l.
\end{equation}
Moreover, for Nadiradze's invariant representatives,
\begin{equation}\label{eq:Hq-stong-quotient}
 H_q(x,y)
 =-\sum_{k,l\geq0}[M(k,l;q;\Ii)]x^ky^l.
\end{equation}
In both sums, we set the terms with $q>k+l+1$ equal to zero. Consequently,
$H_q(x,y)\in\LambdaSt[[x,y]]$.
\end{proposition}

\begin{proof}
Equation \eqref{eq:Hq-stong-double} is Buchstaber's coefficient formula
\cite[Introduction and Theorems~5.10--5.13]{Buc78}.
The differential definition of $H_1$ is given in Theorem~5.6 of the same
paper, and its extension to $H_q$ is stated in the Introduction.
For every admissible triple $(k,l,q)$, Nadiradze's construction
and covering formula give
\[
 [M(k,l;q)]=2[M(k,l;q;\Ii)]
\]
in complex cobordism \cite[Lemma~1 and the subsequent special case, pp.~264--265]{Nad80}. Substitution into
\eqref{eq:Hq-stong-double} and cancellation of $2$, which is valid because
$\OmU$ has no additive torsion, give \eqref{eq:Hq-stong-quotient}.
\end{proof}

We next give an explicit expression for $H_q$ solely in terms of the
logarithm $B$ and its inverse. Besides its use below, this formula makes the
recurrence in $q$ transparent.

\begin{proposition}[(Closed formula for $H_q$)]
\label[proposition]{prop:Hq-closed-form}
For every $q\geq1$,
\begin{equation}\label{eq:Hq-closed-form}
\begin{split}
 H_q(x,y)
 &=\frac{B'(x)B'(y)}{4\sqrt{B(x)B(y)}}
 \Bigl[
  E\!\left(
    \bigl(\sqrt{B(x)}+\sqrt{B(y)}\bigr)^2
   \right)^q
 \\[-0.2em]
 &\hspace{8em}
  -E\!\left(
    \bigl(\sqrt{B(x)}-\sqrt{B(y)}\bigr)^2
   \right)^q
 \Bigr].
\end{split}
\end{equation}
The right-hand side, although written with formal square roots, belongs to
$(\OmU\otimes\Q)[[x,y]]$.
\end{proposition}

\begin{proof}
Put
\[
 X=B(x),\qquad Y=B(y),\qquad U=\sqrt X,\qquad V=\sqrt Y,
\]
and set
\[
 \psi_q(T)=\Psi_q(E(T)).
\]
In the $B$-coordinate, the generalized shift is
\begin{equation}\label{eq:uq-B-coordinate}
 u_q(x,y)
 =\frac12\left(
  \psi_q((U+V)^2)+\psi_q((U-V)^2)
 \right).
\end{equation}
For an arbitrary series $\psi(T)$, the coefficient of $Y$ in
\[
 \frac12\left(
  \psi((\sqrt X+\sqrt Y)^2)
  +\psi((\sqrt X-\sqrt Y)^2)
 \right)
\]
is $\psi'(X)+2X\psi''(X)$. Therefore the equation
$\mathcal D_1\Psi_q=x^q$ becomes
\begin{equation}\label{eq:psi-q-differential}
 \psi_q'(T)+2T\psi_q''(T)=E(T)^q.
\end{equation}
Differentiating \eqref{eq:uq-B-coordinate} first with respect to $U$ and then
with respect to $V$, and using \eqref{eq:psi-q-differential}, gives
\[
 \frac{\partial^2u_q}{\partial U\,\partial V}
 =E((U+V)^2)^q-E((U-V)^2)^q.
\]
Since
\[
 \frac{\partial}{\partial x}
 =\frac{B'(x)}{2U}\frac{\partial}{\partial U},
 \qquad
 \frac{\partial}{\partial y}
 =\frac{B'(y)}{2V}\frac{\partial}{\partial V},
\]
formula \eqref{eq:Hq-closed-form} follows. The difference in square brackets
is divisible by $UV$, and the quotient is invariant under the independent
sign changes $U\mapsto-U$ and $V\mapsto-V$. Hence it is an ordinary series in
$U^2=B(x)$ and $V^2=B(y)$, and therefore in $x$ and $y$.
\end{proof}

\begin{corollary}[{\cite[Theorems~5.12--5.13]{Buc78}}]\label[corollary]{cor:Hq-recurrence-generators}
The series $H_q$ satisfy
\begin{equation}\label{eq:Hq-recurrence}
 H_{q+1}
 =\Theta_1H_q-\Theta_2H_{q-1},
 \qquad q\geq1,
\end{equation}
where $H_0=0$ and $H_1(0,0)=1$. In particular,
\begin{equation}\label{eq:theta-via-H}
 \Theta_1=\frac{H_2}{H_1},
 \qquad
 \Theta_2=\Theta_1^2-\frac{H_3}{H_1}.
\end{equation}
If
\[
 R_H=\mathbb Z[\text{coefficients of }H_1,H_2,H_3],
\]
then
\begin{equation}\label{eq:RH-LambdaSt}
 R_H=\LambdaSt.
\end{equation}
\end{corollary}

\begin{proof}
The difference $z_+^q-z_-^q$ satisfies the standard recurrence determined by
$z_++z_-=\Theta_1$ and $z_+z_-=\Theta_2$. Formula
\eqref{eq:Hq-closed-form} therefore gives \eqref{eq:Hq-recurrence}, and
\eqref{eq:theta-via-H} follows from the cases $q=1,2$. By
\eqref{eq:Hq-stong-quotient}, the coefficients of $H_1,H_2,H_3$ are the
negatives of the classes $[M(k,l;q;\Ii)]$, $q=1,2,3$. These classes generate
$\LambdaSt$ by \cite[Theorem~5.13]{Buc78}, together with
Nadiradze's identification of the halves with quotient classes
\cite[pp.~264--265]{Nad80}. This proves
\eqref{eq:RH-LambdaSt}.
\end{proof}

\subsection{An odd integer-valued polynomial}

Fix $n\geq1$. For positive integers $q$, define
\begin{equation}\label{eq:def-Pnq}
 P_n(q)=[x^{q-1}y^n]H_q(x,y).
\end{equation}
By Proposition~\ref{prop:Hq-stong-expansion},
\begin{equation}\label{eq:Pnq-stong-class}
 P_n(q)=-[M(q-1,n;q;\Ii)]\in\LambdaSt.
\end{equation}
The real dimension of $M(q-1,n;q;\Ii)$ is $4n$, independently of $q$.
The next proposition is the link between these geometric classes and the
coefficient $b_n$.

\begin{proposition}\label[proposition]{prop:Pn-interpolation}
There is a unique odd polynomial
\[
 P_n(T)\in(\LambdaSt\otimes\Q)[T]
\]
of degree at most $2n+1$ whose value at every positive integer $q$ is given by
\eqref{eq:def-Pnq}. It satisfies
\begin{equation}\label{eq:Pn-derivative}
 P_n'(0)=(n+1)b_n.
\end{equation}
\end{proposition}

\begin{proof}
Apply the ramified residue formula, Lemma~\ref{lem:ramified-residue}, first in $x$
and then in $y$ to the changes of variables
\[
 x=E(u^2),\qquad y=E(v^2).
\]
Using \eqref{eq:Hq-closed-form}, we obtain
\begin{equation}\label{eq:Pn-residue}
 P_n(q)
 =\frac14\res_{v=0}\res_{u=0}
 \left(R_+(u,v)^q-R_-(u,v)^q\right)
 E(v^2)^{-n-1}\,\dd u\,\dd v,
\end{equation}
where
\begin{equation}\label{eq:Rpm}
 R_\pm(u,v)
 =\frac{E((u\pm v)^2)}{E(u^2)}.
\end{equation}
The inner residue is taken first, with all expressions expanded $v$-adically.
Since $R_\pm(u,v)=1+O(v)$, we may replace the positive integer $q$ in
\eqref{eq:Pn-residue} by an indeterminate $T$ and use
\[
 R_\pm(u,v)^T
 =\sum_{j\geq0}\binom{T}{j}(R_\pm(u,v)-1)^j.
\]
Furthermore,
\[
 E(v^2)^{-n-1}
 =v^{-2n-2}\left(\frac{v^2}{E(v^2)}\right)^{n+1}.
\]
Thus only terms of $v$-degree at most $2n+1$ in $R_\pm^T$ contribute to the
outer residue. The coefficient of $v^j$ has degree at most $j$ in $T$, so the
result is a polynomial of degree at most $2n+1$.

The $v$-adically continuous changes of variables
\[
 u\longmapsto-u-v
 \qquad\text{and}\qquad
 u\longmapsto-u+v
\]
send $R_+$ and $R_-$ to their inverses, respectively, and reverse the sign of
$\dd u$. Formal invariance of residues therefore gives
\[
 P_n(-T)=-P_n(T).
\]
Hence $P_n$ is odd. Since its values at $1,2,\ldots,n+1$ lie in $\LambdaSt$,
ordinary interpolation also shows that its coefficients lie in
$\LambdaSt\otimes\Q$.

It remains to compute the derivative at zero. Differentiating
\eqref{eq:Pn-residue} gives
\begin{equation}\label{eq:Pn-derivative-residue}
 P_n'(0)
 =\frac14\res_{v=0}\res_{u=0}
 \log\!\left(
  \frac{E((u+v)^2)}{E((u-v)^2)}
 \right)
 E(v^2)^{-n-1}\,\dd u\,\dd v.
\end{equation}
Write $E(w)=w\varepsilon(w)$ with $\varepsilon(w)\in1+w(\OmU\otimes\Q)[[w]]$.
Then
\[
 \log\!\left(
  \frac{E((u+v)^2)}{E((u-v)^2)}
 \right)
 =2\log\!\left(\frac{u+v}{u-v}\right)
  +\log\!\left(
    \frac{\varepsilon((u+v)^2)}{\varepsilon((u-v)^2)}
   \right).
\]
The second summand has no $u^{-1}$ term. In the $v$-adic expansion,
\[
 2\log\!\left(\frac{u+v}{u-v}\right)
 =4\sum_{j\geq0}\frac1{2j+1}\left(\frac vu\right)^{2j+1},
\]
so its $u$-residue is $4v$. Therefore
\begin{equation}\label{eq:Pn-derivative-coefficient}
 P_n'(0)
 =\res_{v=0}vE(v^2)^{-n-1}\,\dd v
 =[w^n]\left(\frac{w}{E(w)}\right)^{n+1}.
\end{equation}
Lagrange inversion for $B=E^{-1}$ gives
\[
 [x^{n+1}]B(x)
 =\frac1{n+1}[w^n]\left(\frac{w}{E(w)}\right)^{n+1}.
\]
Since $[x^{n+1}]B(x)=b_n$, this proves \eqref{eq:Pn-derivative}.
\end{proof}

\subsection{Arithmetic of odd integer-valued polynomials}
\label{sec:odd-polynomial-arithmetic}

For $N\geq1$, put
\[
 L_N=\lcm(1,2,\ldots,N),
 \qquad
 e_n=\frac{L_{2n+2}}2.
\]
Thus $d_n=(n+1)e_n$.

\begin{lemma}[(Odd integer-valued polynomials)]
\label[lemma]{lem:odd-integer-valued}
Let $R$ be a torsion-free commutative ring, and let
\[
 P(T)\in(R\otimes\Q)[T]
\]
be an odd polynomial of degree at most $2n+1$ such that $P(m)\in R$ for every
$m\in\mathbb Z$. Then
\begin{equation}\label{eq:odd-int-derivative}
 e_nP'(0)\in R.
\end{equation}
The factor $e_n$ is optimal: already for $R=\mathbb Z$, if a positive integer
$e$ satisfies $eP'(0)\in\mathbb Z$ for every such polynomial $P$, then
$e_n\mid e$. More precisely, if $m=n+1$, then
\begin{equation}\label{eq:central-difference}
 e_nP'(0)
 =\sum_{q=1}^{m}A_{n,q}P(q),
\end{equation}
where the integers $A_{n,q}$ are given by \eqref{eq:Anq-main}.
\end{lemma}

\begin{proof}
For $0\leq r\leq n$, put
\begin{equation}\label{eq:phi-r-int-valued}
 \phi_r(T)
 =\binom{T+r}{2r+1}
 =\frac{T\prod_{j=1}^{r}(T^2-j^2)}{(2r+1)!}.
\end{equation}
These are odd integer-valued polynomials. Moreover,
\[
 \phi_r(r+1)=1,
 \qquad
 \phi_s(r+1)=0\quad\text{for }s>r.
\]
Consequently, the polynomials $\phi_0,\ldots,\phi_n$ form an $R$-basis of the
$R$-module of odd $R$-valued polynomials of degree at most $2n+1$, and the
evaluation map
\[
 P\longmapsto(P(1),P(2),\ldots,P(n+1))
\]
is an isomorphism onto $R^{n+1}$.

Differentiating \eqref{eq:phi-r-int-valued} at zero gives
\begin{equation}\label{eq:phi-r-derivative}
 \phi_r'(0)
 =(-1)^r\frac{(r!)^2}{(2r+1)!}.
\end{equation}
Set
\[
 \lambda_r=\frac{(2r+1)!}{(r!)^2}.
\]
We claim that $\lambda_r$ divides $e_n$ for every $0\leq r\leq n$.

Let $p$ be an odd prime. Legendre's formula gives
\[
 \nu_p(\lambda_r)
 =\sum_{a\geq1}
 \left(
  \left\lfloor\frac{2r+1}{p^a}\right\rfloor
  -2\left\lfloor\frac r{p^a}\right\rfloor
 \right).
\]
If $r=Ap^a+s$ with $0\leq s<p^a$, the summand is
$\lfloor(2s+1)/p^a\rfloor$, and hence is either $0$ or $1$. Therefore
\[
 \nu_p(\lambda_r)
 \leq\left\lfloor\log_p(2r+1)\right\rfloor
 \leq\left\lfloor\log_p(2n+2)\right\rfloor
 =\nu_p(e_n).
\]
For $p=2$, Legendre's formula yields
\[
 \nu_2(\lambda_r)=s_2(r),
\]
where $s_2(r)$ is the number of ones in the binary expansion of $r$. If
$s_2(r)=k$, then $r\geq2^k-1$, and therefore
\[
 s_2(r)
 \leq\left\lfloor\log_2(r+1)\right\rfloor
 \leq\left\lfloor\log_2(n+1)\right\rfloor
 =\nu_2(e_n).
\]
This proves $\lambda_r\mid e_n$.

We now show that this common multiple is the least possible one. Put
\[
 E_n=\lcm\{\lambda_0,\lambda_1,\ldots,\lambda_n\}.
\]
The divisibility just proved gives $E_n\mid e_n$. Conversely, let $p$ be an
odd prime and put $\alpha=\lfloor\log_p(2n+2)\rfloor$ and
$r=(p^\alpha-1)/2$. Then $r\leq n$, and every summand in Legendre's formula
for $\nu_p(\lambda_r)$ with $1\leq a\leq\alpha$ equals one. All later
summands vanish. Consequently,
\[
 \nu_p(\lambda_r)=\alpha=\nu_p(e_n).
\]
For $p=2$, put $\alpha=\lfloor\log_2(n+1)\rfloor$ and $r=2^\alpha-1$.
Then $r\leq n$ and
\[
 \nu_2(\lambda_r)=s_2(r)=\alpha=\nu_2(e_n).
\]
Thus every prime power dividing $e_n$ occurs in one of the numbers
$\lambda_r$, and therefore $E_n=e_n$. Since
$\phi_r'(0)=(-1)^r/\lambda_r$, any positive integer that makes $P'(0)$
integral for every odd integer-valued polynomial $P$ of degree at most
$2n+1$ must be divisible by all $\lambda_r$, and hence by $e_n$. This proves
the optimality assertion.

Writing $P=\sum_{r=0}^{n}a_r\phi_r$ with $a_r\in R$ and using
\eqref{eq:phi-r-derivative}, we obtain
\[
 e_nP'(0)
 =\sum_{r=0}^{n}
  (-1)^r\frac{e_n}{\lambda_r}a_r\in R.
\]
This proves \eqref{eq:odd-int-derivative}.

It remains to identify the coefficients in terms of the values $P(q)$. Since
$P(T)/T$ is an even polynomial, Lagrange interpolation at
$1^2,2^2,\ldots,m^2$ gives
\[
 P'(0)
 =\sum_{q=1}^{m}
  (-1)^{q+1}\frac{2}{q}
  \frac{\binom{2m}{m-q}}{\binom{2m}{m}}P(q).
\]
Multiplication by $e_n$ gives \eqref{eq:central-difference}. Taking $R=\mathbb Z$,
the preceding basis argument shows that the resulting linear functional on $\mathbb Z^{n+1}$ has
integral coefficients. Hence all $A_{n,q}$ are integers.
\end{proof}

\begin{proof}[Proof of \cref{thm:stong-numerators}]
Apply Lemma~\ref{lem:odd-integer-valued} to the polynomial $P_n(T)$ from
Proposition~\ref{prop:Pn-interpolation}, with $R=\LambdaSt$. By
\eqref{eq:Pn-derivative},
\[
 e_nP_n'(0)
 =e_n(n+1)b_n
 =d_nb_n
 =C_n.
\]
Thus $C_n\in\LambdaSt$. Using \eqref{eq:central-difference} and
\eqref{eq:Pnq-stong-class}, we obtain
\[
 C_n
 =\sum_{q=1}^{n+1}A_{n,q}P_n(q)
 =-\sum_{q=1}^{n+1}A_{n,q}[M(q-1,n;q;\Ii)],
\]
which is exactly \eqref{eq:Cn-stong-combination-main}.
\end{proof}

\subsection{The quaternionic lift and its primitive exact numerator}
\label{sec:sp-lift}

\begin{proof}[Proof of \cref{prop:sp-covering-numerator}]
For each $q$, Nadiradze's double covering \cite{Nad80}
\[
 \pi\colon M(q-1,n;q)\longrightarrow M(q-1,n;q;\Ii)
\]
is compatible with the stable complex structures obtained by forgetting the
stable $\Sp$-structure upstairs and the self-conjugate structure downstairs.
Hence
\begin{equation}\label{eq:Sp-cover-forgetful}
 j([M(q-1,n;q)]_{\Sp})
 =2[M(q-1,n;q;\Ii)].
\end{equation}
Indeed, naturality of the Chern classes and $\deg\pi=2$ give, for every Chern
monomial $c_I$ of top degree,
\[
 \begin{aligned}
 &\left\langle c_I(M(q-1,n;q)),[M(q-1,n;q)]\right\rangle\\
 &\qquad=\left\langle\pi^*c_I(M(q-1,n;q;\Ii)),
                     [M(q-1,n;q)]\right\rangle\\
 &\qquad=2\left\langle c_I(M(q-1,n;q;\Ii)),
                      [M(q-1,n;q;\Ii)]\right\rangle.
 \end{aligned}
\]
Since Chern numbers determine complex cobordism classes, this proves
\eqref{eq:Sp-cover-forgetful}.
Together with \eqref{eq:def-Cn-hat} and
\eqref{eq:Cn-stong-combination-main}, this yields
$j(\widehat C_n)=2C_n$. Injectivity of $j_{\Q}$ then gives
$\widehat C_n=2d_n\widetilde b_n$. Hence
$\widetilde{d_n}\mid2d_n$. Conversely, primitivity of $C_n$ implies that
$(N/d_n)C_n\in\OmU$ only if $d_n\mid N$. Applying $j$ therefore gives
$d_n\mid\widetilde{d_n}$, proving \eqref{eq:dn-Sp-dichotomy}.
\end{proof}

\begin{proof}[Proof of \cref{thm:sp-denominator}]
By \eqref{eq:dn-Sp-dichotomy}, it suffices to show that
$C_n\notin j(\OmSpfree^{-4n})$.
Suppose that $C_n=j([X]_{\Sp})$ for a closed stably quaternionic
manifold $X^{4n}$. For $\pi\colon X\to\mathrm{pt}$ and $0\leq r\leq n$, put
\[
 Y_r:=\pi_!^{\Sp}\bigl(p_r^{\Sp}(\tau_X)\bigr)
 \pmod{\Tors}\in\OmSpfree^{-4(n-r)},
\]
where $\pi_!^{\Sp}$ is the integral Gysin homomorphism in quaternionic
cobordism. The Buchstaber--Novikov symplectic Pontryagin classes are evaluated
on the stable quaternionic
tangent bundle, so this definition is unchanged by adding trivial summands.

Write $\operatorname{ch}_{\U}$ for the Chern--Dold character and
$\mathrm{id}\otimes\Td$ for application of the Todd genus to its cobordism
coefficients. For a complex line bundle $L$ with $x=c_1(L)$, the character
and genus formulas \cite[(3), (45)]{BV24} give
$\operatorname{ch}_{\U}(c_1^{\U}(L))=f(x)$ and
$\Td(f(x))=1-e^{-x}$. The quaternionic splitting principle and Whitney
formula \cite[\S1]{PW18}, with our normalization
$\rho(p_1^{\Sp}(L\oplus\bar L))=c_2^{\U}(L\oplus\bar L)$, consequently give
\[
 (\mathrm{id}\otimes\Td)\operatorname{ch}_{\U}
 \bigl(\rho(p_r^{\Sp}(\tau_X))\bigr)
 =[c^r]\prod_i\bigl(1+c(2-e^{x_i}-e^{-x_i})\bigr),
\]
where $\pm x_i$ are the formal Chern roots of the stable tangent bundle.
Compatibility of $\rho$ with Gysin homomorphisms and the
Riemann--Roch--Grothendieck--Hirzebruch formula in complex cobordism
\cite[Theorem~2.4, formula~(29)]{BV24} now yield
\[
 \begin{aligned}
 \Td(j(Y_r))
 &=\left\langle
   \Td(\tau_X)(\mathrm{id}\otimes\Td)\operatorname{ch}_{\U}
   \bigl(\rho(p_r^{\Sp}(\tau_X))\bigr),[X]
   \right\rangle\\
 &=(-1)^r[c^r]\mathcal R_n(c).
 \end{aligned}
\]
The last equality follows from \eqref{eq:intro-Phi-roots} and
$\Phi_t(C_n)=\mathcal R_n(t(1-t))$. Since a stably quaternionic manifold is
spin with $c_1=0$, its Todd genus equals its ordinary $\widehat A$-genus.
Thus \cite[Corollary~2(ii), p.~279]{AH59} implies that
$[c^r]\mathcal R_n(c)$ is even whenever $n-r$ is odd.

Choose
\[
 q=2^{\lfloor\log_2(n+1)\rfloor},
 \qquad r=n+1-q.
\]
The $p=2$ calculation in the proof of \cref{lem:arithmetic} gives
$\nu_2(d_nr_{n+1,q})=0$. Hence
\eqref{eq:primitive-index-polynomial} shows that
$[c^r]\mathcal R_n(c)=(-1)^r d_nr_{n+1,q}$ is odd. But
$n-r=q-1$ is odd, a contradiction. Therefore
$C_n\notin j(\OmSpfree^{-4n})$, and the dichotomy forces
$\widetilde{d_n}=2d_n$.

Finally, if $\widehat C_n=aY$ with an integer $a>1$, then
$2C_n=a\,j(Y)$. Primitivity of $C_n$ forces $a=2$, and hence
$j(Y)=C_n$, contradicting the result just proved. Thus
$\widetilde C_n=\widehat C_n$ is primitive.
\end{proof}

\begin{remark}[(The first quaternionic denominator)]
\label[remark]{rem:first-Sp-denominator}
The second case occurs already for $n=1$. By
\eqref{eq:C1-C2-Enriques}, $C_1$ is represented by an Enriques surface $S$,
whose signature is
\[
 \sigma(C_1)
 =\frac{c_1^2[S]-2c_2[S]}{3}
 =-8.
\]
Every stably quaternionic four-manifold is spin, and its signature is divisible
by $16$ \cite{Rok52,Och81}. Hence $C_1\notin j(\OmSpfree^{-4})$. Since
$d_1=12$, we obtain
\[
 \widetilde{d_1}=24=2d_1,
\]
agreeing with \cref{thm:sp-denominator}.
\end{remark}

\subsection{\texorpdfstring{$\Theta_1$}{Theta1} on the diagonal and the coefficient ring \texorpdfstring{$\Lambda$}{Lambda}}
\label{sec:coefficient-ring}

Put
\[
 \Lambda_{(2)}=\Lambda\otimes\Ztwo.
\]
For a torsion-free commutative $\Ztwo$-algebra $R$ and an integer $a$, we write $2^aR$
for the corresponding fractional $R$-submodule of $R[1/2]$.

\begin{proposition}[($\Theta_1$ on the diagonal)]
\label[proposition]{prop:diagonal-doubling}
One has
\begin{equation}\label{eq:diagonal-doubling}
 \Theta_1(x,x)=E(4B(x)),
 \qquad
 B\bigl(\Theta_1(x,x)\bigr)=4B(x).
\end{equation}
In particular, $\Theta_1(x,x)=4x+O(x^2)$. Moreover,
\begin{equation}\label{eq:diagonal-parity}
 [x^k]\Theta_1(x,x)\in2\Lambda
 \qquad\text{for every odd }k>1.
\end{equation}
\end{proposition}

\begin{proof}
On the diagonal $y=x$, the two branches in \eqref{eq:zpm-B-coordinate} are
\[
 z_+(x,x)=E(4B(x)),
 \qquad
 z_-(x,x)=E(0)=0.
\]
Since $\Theta_1=z_++z_-$, this proves \eqref{eq:diagonal-doubling}.

The strong unit identity gives $\Theta_1(x,0)=2x$, and symmetry gives
$\Theta_1(x,y)=\Theta_1(y,x)$. Hence we may write
\[
 \Theta_1(x,y)
 =2x+2y+\sum_{i,j\geq1}\theta_{ij}x^iy^j,
 \qquad
 \theta_{ij}=\theta_{ji}\in\Lambda.
\]
For $k>1$,
\[
 [x^k]\Theta_1(x,x)=\sum_{i+j=k}\theta_{ij}.
\]
If $k$ is odd, no term with $i=j$ occurs, and the remaining terms are paired
by $(i,j)\leftrightarrow(j,i)$. Thus $[x^k]\Theta_1(x,x)\in2\Lambda$.
\end{proof}

For $N\geq1$, put
\begin{equation}\label{eq:def-kappa-two}
 \kappa_2(N)=\nu_2(N)+\left\lfloor\log_2N\right\rfloor.
\end{equation}

\begin{lemma}[(A dyadic estimate for powers)]
\label[lemma]{lem:dyadic-powers}
Let $R$ be a torsion-free commutative $\Ztwo$-algebra and let
\[
 F(x)=4x+\sum_{k\geq2}f_kx^k\in R[[x]]
\]
satisfy $f_k\in2R$ for every odd $k>1$. Then, for $1\leq M\leq N$,
\begin{equation}\label{eq:dyadic-power-estimate}
 [x^N]F(x)^M
 \in
 2^{\max\{0,\,2+\kappa_2(M)-\kappa_2(N)\}}R.
\end{equation}
\end{lemma}

\begin{proof}
Expand the coefficient as a sum over ordered compositions
\[
 k_1+\cdots+k_M=N,
 \qquad k_i\geq1,
\]
and group the summands into orbits under cyclic rotation. Fix one orbit and
let $d$ be its cardinality. The corresponding $M$-tuple is obtained by
repeating a primitive block of length $d$ exactly $\rho=M/d$ times. If $S$ is
the sum of the entries in the primitive block, then
\[
 M=\rho d,
 \qquad
 N=\rho S,
\]
and therefore
\begin{equation}\label{eq:orbit-vtwo}
 \nu_2(d)=\nu_2(M)-\nu_2(N)+\nu_2(S).
\end{equation}

Let $\ell$ be the number of entries equal to $1$ and $\ell_{\mathrm{odd}}$ the
number of odd entries greater than $1$, both counted in the full $M$-tuple.
Since $f_1=4$ and
$f_k\in2R$ for odd $k>1$, the product attached to the tuple belongs to
$2^{2\ell+\ell_{\mathrm{odd}}}R$. The sum over its cyclic orbit therefore belongs to
\[
 2^{\nu_2(d)+2\ell+\ell_{\mathrm{odd}}}R.
\]
Put
\[
 \Delta
 =\left\lfloor\log_2N\right\rfloor
  -\left\lfloor\log_2M\right\rfloor\geq0.
\]
Using \eqref{eq:orbit-vtwo}, the difference between the last exponent and
$2+\kappa_2(M)-\kappa_2(N)$ is
\begin{equation}\label{eq:orbit-difference}
 \nu_2(S)+2\ell+\ell_{\mathrm{odd}}-2+\Delta.
\end{equation}
If $\Delta\geq2$, this is nonnegative. If $\Delta=1$, then either at least one entry is
odd, so $2\ell+\ell_{\mathrm{odd}}\geq1$, or all entries are even, so $S$ is even and
$\nu_2(S)\geq1$. Finally, if $\Delta=0$, then $N<2M$. Hence some entry equals
$1$, and $\ell\geq1$. Thus \eqref{eq:orbit-difference} is always nonnegative.
Every orbit has the divisibility asserted in
\eqref{eq:dyadic-power-estimate}, and so does their sum.
\end{proof}

\begin{lemma}[(Dyadic linearization)]
\label[lemma]{lem:dyadic-linearization}
Let $R$ and $F$ satisfy the assumptions of \cref{lem:dyadic-powers}. Suppose
that
\[
 B_0(x)=\sum_{N\geq1}h_Nx^N
 \in R[1/2][[x]],
 \qquad
 h_1=1,
\]
satisfies
\begin{equation}\label{eq:abstract-schroeder}
 B_0(F(x))=4B_0(x).
\end{equation}
Then
\begin{equation}\label{eq:dyadic-linearization-conclusion}
 2^{\kappa_2(N)}h_N\in R
 \qquad\text{for every }N\geq1.
\end{equation}
\end{lemma}

\begin{proof}
We argue by induction on $N$. The assertion is clear for $N=1$. For $N\geq2$,
comparison of the coefficient of $x^N$ in \eqref{eq:abstract-schroeder} gives
\begin{equation}\label{eq:linearization-recurrence}
 (4^N-4)h_N
 =-\sum_{M=1}^{N-1}h_M[x^N]F(x)^M.
\end{equation}
By the induction hypothesis and \cref{lem:dyadic-powers}, the $M$th summand on
the right belongs to
\[
 2^{-\kappa_2(M)
 +\max\{0,\,2+\kappa_2(M)-\kappa_2(N)\}}R
 \subset 2^{2-\kappa_2(N)}R.
\]
Hence the right-hand side of \eqref{eq:linearization-recurrence} belongs to
$2^{2-\kappa_2(N)}R$. Since
\[
 4^N-4=4(4^{N-1}-1)
\]
and $4^{N-1}-1$ is a unit in $\Ztwo$, division by $4^N-4$ gives
$h_N\in2^{-\kappa_2(N)}R$. This is
\eqref{eq:dyadic-linearization-conclusion}.
\end{proof}

\begin{proposition}[(Inclusion in the coefficient ring at $2$)]
\label[proposition]{prop:Cn-Lambda-two-local}
For every $n\geq1$,
\begin{equation}\label{eq:two-local-bn}
 2^{\nu_2(d_n)}b_n\in\Lambda_{(2)}.
\end{equation}
Consequently,
\[
 C_n=d_nb_n\in\Lambda_{(2)}.
\]
\end{proposition}

\begin{proof}
By \cref{thm:stong-numerators,eq:ring-hierarchy-main}, the coefficients of
$B(x)$ belong to $\Lambda\otimes\Q=\Lambda_{(2)}[1/2]$. Apply
\cref{lem:dyadic-linearization} with
\[
 R=\Lambda_{(2)},
 \qquad
 F(x)=\Theta_1(x,x),
 \qquad
 B_0(x)=B(x).
\]
The hypotheses on $F$ are supplied by \cref{prop:diagonal-doubling}, and
\eqref{eq:abstract-schroeder} is precisely
$B\bigl(\Theta_1(x,x)\bigr)=4B(x)$. Since $h_{n+1}=b_n$, we obtain
\[
 2^{\kappa_2(n+1)}b_n\in\Lambda_{(2)}.
\]
Moreover,
\begin{align*}
 \nu_2(d_n)
 &=\nu_2(n+1)-1
   +\nu_2\bigl(\lcm(1,2,\ldots,2n+2)\bigr)\\
 &=\nu_2(n+1)+\left\lfloor\log_2(n+1)\right\rfloor\\
 &=\kappa_2(n+1).
\end{align*}
This proves \eqref{eq:two-local-bn}. Multiplication by the odd integer
$d_n/2^{\nu_2(d_n)}$ gives $C_n\in\Lambda_{(2)}$.
\end{proof}

\begin{proof}[Proof of \cref{thm:Cn-Lambda}]
By \cref{thm:stong-numerators,eq:ring-hierarchy-main},
\[
 C_n\in\LambdaSt
 \subset\St
 \subset\widetilde\Lambda
 \subset\Lambda\!\left[\frac12\right].
\]
On the other hand, \cref{prop:Cn-Lambda-two-local} gives
$C_n\in\Lambda_{(2)}$. In degree $-4n$, the group
$\Omega_{\Uu}^{-4n}$ is free abelian of finite rank, so its subgroup
$\Lambda^{-4n}$ is free. Hence, inside
$\Lambda^{-4n}\otimes\Q$,
\[
 \Lambda^{-4n}\!\left[\frac12\right]
 \cap
 \bigl(\Lambda^{-4n}\otimes\Ztwo\bigr)
 =\Lambda^{-4n}.
\]
Thus $C_n\in\Lambda^{-4n}$.
\end{proof}

\begin{remark}[(The first numerator classes)]
Let
\[
 \theta_{ij}=[x^iy^j]\Theta_1(x,y).
\]
Direct expansion gives
\[
 C_1=-\theta_{11},
 \qquad
 C_2=C_1^2-3\theta_{21},
\]
and
\[
 C_3
 =-25\theta_{31}-2\theta_{22}
  +6C_1C_2-3C_1^3.
\]
Thus $C_1,C_2,C_3$ already lie in the smaller coefficient ring generated by
$\Theta_1$. Theorem~\ref{thm:Cn-Lambda} gives the uniform statement in all degrees
in $\Lambda$, while Theorem~\ref{thm:stong-numerators} additionally gives
explicit integral expressions in quotient Stong manifold classes.
\end{remark}

\section{Chern numbers of the primitive numerators}\label{sec:chern-divisibility}

In this section, we prove that all Chern numbers of
$C_n\in\Lambda^{-4n}\subset\LambdaSt^{-4n}$ are divisible by
$d_n=c_{2n}(C_n)$, even though $C_n$ is primitive in
$\Omega_{\Uu}^{-4n}$.

We use a genus that records all Chern numbers simultaneously. Let
\[
 Q_{\boldsymbol u}(z)=1+u_1z+u_2z^2+\cdots
 \in\mathbb Z[u_1,u_2,\ldots][[z]].
\]
For a stably complex manifold $M$ with formal Chern roots
$x_1,\ldots,x_r$, put
\begin{equation}\label{eq:universal-chern-genus}
 \mathcal C_{\boldsymbol u}(M)
 =\left\langle\prod_{i=1}^rQ_{\boldsymbol u}(x_i),[M]\right\rangle.
\end{equation}
The coefficient of $u_1^{m_1}u_2^{m_2}\cdots$ in
\eqref{eq:universal-chern-genus} is $m_\lambda[M]$, where
$\lambda=(1^{m_1}2^{m_2}\cdots)$. Since products of ordinary Chern classes
are integral linear combinations of the $m_\lambda$, this genus also
records their integrality and common divisors. The exponential of
$\mathcal C_{\boldsymbol u}$ is
\[
 f_{\boldsymbol u}(z)=\frac{z}{Q_{\boldsymbol u}(z)}.
\]

Let $B_{\boldsymbol u}(x)$ be the logarithm of the corresponding two-valued
formal group. Then
\begin{equation}\label{eq:chern-genus-two-valued-exponential}
 B_{\boldsymbol u}^{-1}(-z^2)
 =f_{\boldsymbol u}(z)f_{\boldsymbol u}(-z)
 =-\frac{z^2}{Q_{\boldsymbol u}(z)Q_{\boldsymbol u}(-z)}.
\end{equation}
The product $Q_{\boldsymbol u}(z)Q_{\boldsymbol u}(-z)$ is an even series with
constant term $1$ and integral coefficients. Therefore
\[
 B_{\boldsymbol u}^{-1}(x)
 \in\mathbb Z[u_1,u_2,\ldots][[x]].
\]
The compositional inverse of a strict integral series is again integral, so
\begin{equation}\label{eq:chern-genus-logarithm-integral}
 B_{\boldsymbol u}(x)
 \in\mathbb Z[u_1,u_2,\ldots][[x]].
\end{equation}
Applying the genus $\mathcal C_{\boldsymbol u}$ to the universal identity \eqref{eq:def-B}, we get
\begin{equation}\label{eq:chern-genus-bn}
 \mathcal C_{\boldsymbol u}(b_n)
 =[x^{n+1}]B_{\boldsymbol u}(x)
 \in\mathbb Z[u_1,u_2,\ldots].
\end{equation}
Consequently, every Chern number of the rational cobordism class $b_n$ is an
integer. Since $C_n=d_nb_n$, every Chern number of $C_n$ is divisible by
$d_n$.

To compute the top Chern number, specialize $u_1=u$ and $u_j=0$ for $j\geq2$,
so that $Q_{\boldsymbol u}(z)=1+uz$.
Then \eqref{eq:chern-genus-two-valued-exponential} becomes
\[
 B_u^{-1}(x)=\frac{x}{1+u^2x},
 \qquad
 B_u(x)=\frac{x}{1-u^2x}
       =x+\sum_{n\geq1}u^{2n}x^{n+1}.
\]
The coefficient of $u^{2n}$ in the corresponding genus is the top Chern
number. Hence
\[
 c_{2n}(b_n)=1,
 \qquad
 c_{2n}(C_n)=d_n.
\]

Since all Chern numbers of $C_n$ are divisible by $d_n$ and the top Chern
number equals $d_n$, their greatest common divisor is $d_n$. This proves
\cref{thm:chern-divisibility}.

\section{Odd genera and characteristic classes on the Stong ring}\label{sec:odd-genera}

This section studies the odd genera $\Bc_N$ and their applications to the
logarithm $B(x)$ of the universal two-valued formal group. We first prove that
the universal odd genus has no $2$-primary denominators on the Stong ring. We
then determine the exact odd-primary denominators in $b_n$, identify the
precise range in which the finite truncations are integral on $\St$, explain
why oddification does not change a genus on this ring, and finish with
undecomposable classes in $\Lambda$ that distinguish the Buchstaber genus
from the pair $(\Och,\chi)$.

\subsection{Two-local integrality of the universal odd genus}
\label{sec:two-local-integrality}

Let $B_{\Bc,\infty}(x)$ be the logarithm of the two-valued formal group obtained
from \eqref{eq:bc-infty-exponential}, and let
\[
 \Gamma(x)=B_{\Bc,\infty}^{-1}(x).
\]

\begin{proof}[Proof of \cref{thm:bc-infty-two-local}]
Put
\[
 A(x)=1+a_1x+a_2x^2+\cdots
\]
and set $s=-t^2$. Since $f_{\Bc,\infty}$ is odd, the modulus square identity is
\[
 \Gamma(-t^2)=-f_{\Bc,\infty}(t)^2.
\]
Differentiating this identity and using
\eqref{eq:bc-infty-exponential} gives
\begin{equation}\label{eq:Gamma-two-local-differential}
 s\bigl(\Gamma'(s)\bigr)^2
 =\Gamma(s)A(\Gamma(s)).
\end{equation}
Write
\[
 \Gamma(s)=s+\sum_{n\geq2}\gamma_ns^n.
\]
Comparison of the coefficient of $s^n$ in
\eqref{eq:Gamma-two-local-differential} gives
\[
 (2n-1)\gamma_n
 \in
 \mathbb Z[a_1,\ldots,a_{n-1},
            \gamma_2,\ldots,\gamma_{n-1}].
\]
The integer $2n-1$ is odd. Induction therefore yields
\begin{equation}\label{eq:Gamma-two-local}
 \Gamma(s)\in\Ztwo[a_1,a_2,\ldots][[s]].
\end{equation}
The compositional inverse of a strict series over a ring is again defined over
that ring, and hence
\begin{equation}\label{eq:B-infty-two-local}
 B_{\Bc,\infty}(s)
 \in\Ztwo[a_1,a_2,\ldots][[s]].
\end{equation}
Equivalently, if $Q(t)=t/f_{\Bc,\infty}(t)$ and
$K(s)=s/\Gamma(s)$, then
\[
 Q(t)Q(-t)=K(-t^2),
 \qquad
 K(s)\in\Ztwo[a_1,a_2,\ldots][[s]].
\]

It remains to pass from the logarithm to the coefficient ring of the two-valued
law. Introduce independent variables $U,V$ and put
\[
 Z_\pm(U,V)=\Gamma\bigl((U\mathbin{\pm}V)^2\bigr).
\]
By \eqref{eq:Gamma-two-local}, the product $Z_+Z_-$ is invariant under the
independent sign changes $U\mapsto-U$ and $V\mapsto-V$, and hence belongs to
$\Ztwo[a_1,a_2,\ldots][[U^2,V^2]]$. Furthermore,
\[
 \frac{Z_+(U,V)+Z_-(U,V)}{2}
 \in\Ztwo[a_1,a_2,\ldots][[U^2,V^2]],
\]
because in every pair
$(U+V)^{2r}+(U-V)^{2r}$ the odd powers cancel and the remaining coefficients
have a common factor $2$. We may therefore substitute
\[
 U^2=B_{\Bc,\infty}(x),
 \qquad
 V^2=B_{\Bc,\infty}(y),
\]
using \eqref{eq:B-infty-two-local}. The resulting series are respectively the
images of $\Theta_2(x,y)$ and $\Theta_1(x,y)/2$. Thus $\Bc_\infty$ sends every
coefficient of $\Theta_1/2$ and $\Theta_2$ to
$\Ztwo[a_1,a_2,\ldots]$. Since
$\St=\widetilde\Lambda\cap\OmU$ and $\widetilde\Lambda$ is generated by these
coefficients, the required inclusion follows. Every $\Bc_N$ is obtained by
putting $a_j=0$ for $j>N$, which proves the final assertion.
\end{proof}

\subsection{Exact odd denominators}
\label{sec:odd-denominators}

The logarithm associated with $\Bc_N$ satisfies a differential identity that
excludes all $2$-primary factors from the denominators of its specialized
coefficients $\beta_n=\Bc_N(b_n)$. The remaining odd-primary lower bound is
obtained by comparison with the index genus $\Phi_t$.

\begin{proof}[Proof of \cref{thm:bc-odd-denominator}]
Put
\[
 A_N(x)=1+a_1x+\cdots+a_Nx^N,
 \qquad
 G_N(z)=\int_0^z A_N(s^2)^{-1/2}\,\dd s.
\]
By \eqref{eq:bcN-two-valued-logarithm},
\[
 B_{\Bc,N}(x)=G_N(\sqrt{x})^2.
\]
Write
\[
 B_{\Bc,N}(x)=x+\sum_{r\geq1}\beta_r x^{r+1}.
\]
The coefficient $\beta_r$ depends only on $a_1,\ldots,a_r$. Differentiating
$G_N(\sqrt{x})^2$ gives
\begin{equation}\label{eq:bc-differential-equation}
 xA_N(x)\bigl(B_{\Bc,N}'(x)\bigr)^2=B_{\Bc,N}(x).
\end{equation}
Comparison of the coefficients of $x^{r+1}$ in
\eqref{eq:bc-differential-equation} gives
\begin{equation}\label{eq:bc-recursion-denominator}
 (2r+1)\beta_r
 \in\mathbb Z[a_1,\ldots,a_r,\beta_1,\ldots,\beta_{r-1}].
\end{equation}
Induction on $r$ therefore shows that
\begin{equation}\label{eq:bc-no-two-denominators}
 \beta_r\in\Ztwo[a_1,\ldots,a_r].
\end{equation}
Thus no power of $2$ occurs in the denominator of $\beta_r$.

Write
\begin{equation}\label{eq:def-rho-j}
 A_N(x)^{-1/2}=\sum_{j\geq0}\rho_jx^j.
\end{equation}
The binomial expansion gives
\[
 \rho_j\in\mathbb Z\left[\frac12,a_1,\ldots,a_N\right].
\]
Let $m=n+1$ and take $N\geq n$. Since
\[
 G_N(z)=\sum_{j\geq0}\frac{\rho_j}{2j+1}z^{2j+1},
\]
we obtain
\begin{align}
 \beta_n
 &=\sum_{j=0}^{m-1}
   \frac{\rho_j\rho_{m-1-j}}
        {(2j+1)(2m-2j-1)}\notag\\
 &=\frac1m\sum_{j=0}^{m-1}
   \frac{\rho_j\rho_{m-1-j}}{2j+1}.
 \label{eq:bc-coefficient-formula}
\end{align}
The second equality follows by pairing the summands with indices $j$ and
$m-1-j$.

Let $p$ be an odd prime. The coefficients $\rho_j$ are $p$-integral, and
\eqref{eq:bc-coefficient-formula} shows that the $p$-adic denominator of every
coefficient of $\beta_n$ is at most
\[
 p^{\nu_p(m)+\lfloor\log_p(2m-1)\rfloor}.
\]
This is exactly the $p$-primary part of
\[
 \delta_n=m_{\mathrm{odd}}\lcm(1,3,5,\ldots,2m-1),
 \qquad
 m_{\mathrm{odd}}=\frac{m}{2^{\nu_2(m)}}.
\]
Together with \eqref{eq:bc-no-two-denominators}, this proves
\begin{equation}\label{eq:bc-upper-odd}
 \delta_n\beta_n\in\mathbb Z[a_1,\ldots,a_n].
\end{equation}

It remains to prove that no odd factor can be removed. We use the index
genus $\Phi_t$ from \cref{sec:genus}. Put $c=t(1-t)$ and define the odd
strict series
\begin{equation}\label{eq:oddification-Phi}
 \widetilde f_{\Phi,t}(z)
 =\sqrt{-f_{\Phi,t}(z)f_{\Phi,t}(-z)},
 \qquad
 \widetilde f_{\Phi,t}(z)=z+O(z^3).
\end{equation}
Since $\widetilde f_{\Phi,t}$ is odd,
\[
 \widetilde f_{\Phi,t}(z)\widetilde f_{\Phi,t}(-z)
 =f_{\Phi,t}(z)f_{\Phi,t}(-z).
\]
Hence its modulus square construction has the same logarithm
$B_{\Phi,t}(x)$ as $\Phi_t$. From the formula for
$f_{\Phi,t}(z)f_{\Phi,t}(-z)$ in the proof of \cref{lem:specialized-B},
direct differentiation gives
\begin{equation}\label{eq:oddification-equation}
 \bigl(\widetilde f_{\Phi,t}'(z)\bigr)^2
 =\bigl(1-c\widetilde f_{\Phi,t}(z)^2\bigr)^2
  \left(1+\left(\frac14-c\right)\widetilde f_{\Phi,t}(z)^2\right).
\end{equation}
Consequently, comparison with \eqref{eq:bcN-exponential} gives the
specialization
\begin{equation}\label{eq:bc-to-Phi-specialization}
 \begin{aligned}
 a_1&\longmapsto 3c-\frac14,\\
 a_2&\longmapsto 3c^2-\frac c2,\\
 a_3&\longmapsto c^3-\frac{c^2}{4},\\
 a_j&\longmapsto0,\qquad j\geq4,
 \end{aligned}
\end{equation}
which sends $\beta_n$ to $[x^{n+1}]B_{\Phi,t}(x)=\Phi_t(b_n)$ for
$n\geq3$. For every odd prime $p$, this specialization takes values in
$\mathbb Z_{(p)}[t]$. By \cref{cor:primitive-polynomial,lem:substitution},
\[
 d_n\Phi_t(b_n)=\Phi_t(C_n)=\mathcal R_n\bigl(t(1-t)\bigr)\in\mathbb Z[t]
\]
is primitive.

For $n=1,2$, both assertions follow directly from the displayed formulas
for $\beta_1$ and $\beta_2$, so suppose that $n\geq3$. For every odd prime
$p$, we have
\[
 \nu_p(d_n)=\nu_p(\delta_n),
\]
and hence $\delta_n/d_n$ is a unit in $\mathbb Z_{(p)}$. Under
\eqref{eq:bc-to-Phi-specialization},
\[
 \delta_n\beta_n\longmapsto
 \delta_n\Phi_t(b_n)=\frac{\delta_n}{d_n}\,\Phi_t(C_n).
\]
The polynomial on the right is not divisible by $p$, since $\Phi_t(C_n)$
is primitive and $\delta_n/d_n$ is a $p$-adic unit. Consequently,
$\delta_n\beta_n$ is not divisible by $p$. In particular, whenever
$p\mid\delta_n$, the polynomial $(\delta_n/p)\beta_n$ is not integral.
Thus for every odd prime $p$, the exact $p$-primary denominator of
$\beta_n$ is $p^{\nu_p(\delta_n)}$. It follows that $\delta_n$ is the exact
common denominator of $\beta_n$, and that $\delta_n\beta_n$ is not divisible
by any odd prime.

It remains only to exclude the possibility that all coefficients of \(\delta_n\beta_n\) are even, since the preceding specialization argument applies only at odd primes. The coefficient of \(a_n\) in \(\beta_n\) is \([a_n]\beta_n=-\frac{1}{2n+1}\).
Indeed, it comes from the term \(-a_nz^{2n}/2\) in \(A_n(z^2)^{-1/2}\). Therefore
\[
[a_n]\delta_n\beta_n = -\frac{\delta_n}{2n+1}.
\]
Since \(2n+1\mid\delta_n\) and \(\delta_n\) is odd, this coefficient is an odd integer. Hence \(\delta_n\beta_n\) is not divisible by \(2\). Together with the preceding argument, which excludes all odd primes, this shows that \(\delta_n\beta_n\) is primitive. This proves \cref{thm:bc-odd-denominator}.
\end{proof}

\subsection{Integral truncations of the universal odd genus}
\label{sec:integral-truncations}

We now show that $\Bc_{N}$ takes integral values on the Stong ring $\St$ if and only if $N \leq 3$ and thus prove \cref{thm:bc-stong-threshold}.

For $N=3$, the genus $\Bc_3$ coincides with the Buchstaber genus. The
associated two-valued law is defined by the polynomial (see \cite{Buc90,Kor26})
\begin{equation}\label{eq:buchstaber-polynomial}
 \mathcal B_{\boldsymbol{a}}(z;x,y)
 =\bigl(x+y+z-a_2xyz\bigr)^2
 -4\bigl(1+a_3xyz\bigr)
  \bigl(xy+xz+yz+a_1xyz\bigr).
\end{equation}
Put
\begin{equation}\label{eq:def-A-N-bc}
\begin{aligned}
 Q(x,y)
 &=1-2a_2xy-4a_3xy(x+y)
   +(a_2^2-4a_1a_3)x^2y^2,\\
 P(x,y)
 &=x+y+2a_1xy+a_2xy(x+y)+2a_3x^2y^2.
\end{aligned}
\end{equation}
Then
\begin{equation}\label{eq:buchstaber-polynomial-monic-data}
 \mathcal B_{\boldsymbol{a}}(z;x,y)
 =Q(x,y)z^2-2P(x,y)z+(x-y)^2.
\end{equation}
After division by $Q(x,y)$, comparison with
\eqref{universal_two_valued_formal_group} gives
\begin{equation}\label{eq:bc-images-theta}
 \Bc_3\left(\frac{\Theta_1(x,y)}2\right)
 =\frac{P(x,y)}{Q(x,y)},
 \qquad
 \Bc_3\bigl(\Theta_2(x,y)\bigr)
 =\frac{(x-y)^2}{Q(x,y)}.
\end{equation}
Both $P$ and $Q$ belong to
$\mathbb Z[2a_1,a_2,2a_3][x,y]$, and $Q(0,0)=1$. Hence the right-hand sides of
\eqref{eq:bc-images-theta} have coefficients in
$\mathbb Z[2a_1,a_2,2a_3]$. Since $\widetilde\Lambda$ is generated by the
coefficients of $\Theta_1/2$ and $\Theta_2$, and
$\St=\widetilde\Lambda\cap\OmU$ by
\cite[Theorems~4.1, 4.4 and Lemma~6.6]{Buc78}, we obtain
\[
 \Bc_3(\St)\subset\mathbb Z[2a_1,a_2,2a_3].
\]
The cases $N=2$ and $N=1$ follow by the specializations $a_3=0$ and
$a_2=a_3=0$, respectively.

We next prove failure of integrality for $N>3$. Let
\[
 \theta_{32}=[x^3y^2]\Theta_1(x,y).
\]
This class belongs to $\St$, because the coefficients of $\Theta_1$ lie in
$\widetilde\Lambda\cap\OmU$. Fix $N>3$ and specialize $a_j=0$ for $j\ne4$.
The resulting two-valued logarithm and its inverse are
\[
 B_4(x)
 =\left(\int_0^{\sqrt{x}}\frac{\dd t}{\sqrt{1+a_4t^8}}\right)^2
 =x-\frac{a_4}{9}x^5+O(x^9),
\]
\[
 E_4(x)=B_4^{-1}(x)
 =x+\frac{a_4}{9}x^5+O(x^9).
\]
The first coefficient of the corresponding two-valued law is
\[
 \Theta_{1,4}(x,y)
 =E_4\!\left((\sqrt{B_4(x)}+\sqrt{B_4(y)})^2\right)
  +E_4\!\left((\sqrt{B_4(x)}-\sqrt{B_4(y)})^2\right).
\]
Modulo terms of total degree at least $6$, this gives
\begin{equation}\label{eq:theta14-degree-five}
 \Theta_{1,4}(x,y)
 =2x+2y+\frac{2a_4}{9}
  \sum_{j=1}^{4}\binom{10}{2j}x^{5-j}y^j.
\end{equation}
Therefore
\[
\left. \Bc_N(\theta_{32})\right|_{a_j=0,\ j\ne4} = [x^3y^2]\Theta_{1,4}(x,y) =\frac{140}{3}a_4,
\]
which cannot be the specialization of a polynomial in
$\mathbb Z[a_1,\ldots,a_N]$. This proves
\cref{thm:bc-stong-threshold}.

\subsection{Oddification of Hirzebruch genera}
\label{sec:oddification}

Let $j\colon\OmSp/\Tors\to\OmU$ be the forgetful map. The Stong ring is the
$2$-primary saturation of its image:
\[
 \St=\operatorname{Sat}_2\bigl(j(\OmSp/\Tors)\bigr)
 =j(\OmSp/\Tors)\!\left[\frac12\right]\cap\OmU.
\]
It is therefore enough to compare the two genera on quaternionic
representatives, where the Chern roots occur in pairs $\pm x_i$ and the
characteristic class can be written in the Pontryagin roots $x_i^2$.

\begin{proof}[Proof of \cref{thm:oddification-stong}]
Put
\begin{equation}\label{eq:oriented-series-Q}
 P_Q(u)=\sqrt{Q(\sqrt{u})Q(-\sqrt{u})}\in R[[u]],\qquad P_Q(0)=1.
\end{equation}
The multiplicative sequence determined by $P_Q$ defines an oriented genus
\begin{equation}\label{eq:oriented-extension-Q}
 \varphi_Q^{\mathrm{SO}}(M)
 =\left\langle\prod_{i=1}^rP_Q(u_i),[M]\right\rangle,
\end{equation}
where $u_1,\ldots,u_r$ are the formal Pontryagin roots of the oriented
manifold $M$. After inverting $2$, this is the extension obtained from the
composite \cite{Sto67}
\[
 \Omega_*^{\mathrm{SO}}\longrightarrow
 \Omega_*^{\mathrm{SO}}\!\left[\frac12\right]
 \cong
 \Omega_*^{\mathrm{Sp}}\!\left[\frac12\right]
 \longrightarrow
 \Omega_*^{\mathrm U}\!\left[\frac12\right]
 \xrightarrow{\ \varphi_Q\ } R.
\]

Let $Y$ be a stably quaternionic manifold. The formal Chern roots of its
underlying stably complex tangent bundle occur in pairs
\[
 x_1,-x_1,\ldots,x_r,-x_r,
\]
and
\[
 p(\tau_Y)=\prod_{i=1}^r(1+x_i^2)^2.
\]
Thus each $u_i=x_i^2$ occurs twice among the ordinary Pontryagin roots. Hence
\begin{align*}
 \varphi_Q(Y)
 &=\left\langle
   \prod_{i=1}^r Q(x_i)Q(-x_i),[Y]
  \right\rangle\\
 &=\left\langle
   \prod_{i=1}^r P_Q(u_i)^2,[Y]
  \right\rangle
 =\varphi_Q^{\mathrm{SO}}(Y).
\end{align*}
The series $\widetilde Q$ is even, and
\[
 \widetilde Q(x_i)^2
 =Q(x_i)Q(-x_i)
 =P_Q(u_i)^2.
\]
Therefore the same calculation gives
\[
 \varphi_{\widetilde Q}(Y)
 =\left\langle
   \prod_{i=1}^r\widetilde Q(x_i)\widetilde Q(-x_i),[Y]
  \right\rangle
 =\varphi_Q^{\mathrm{SO}}(Y).
\]
Thus $\varphi_Q$ and $\varphi_{\widetilde Q}$ agree on the image of
quaternionic cobordism in complex cobordism.

Now let $X\in\St$. Since $\St$ is the $2$-primary saturation of that image, there are
an integer $k\geq0$ and a class $Y\in\OmSp/\Tors$ such that
\[
 2^kX=j(Y)\quad\text{in }\OmU.
\]
Consequently,
\[
 2^k\bigl(\varphi_Q(X)-\varphi_{\widetilde Q}(X)\bigr)=0.
\]
Since $2^k$ is a unit in the $\mathbb Q$-algebra $R$, the required equality
follows.
\end{proof}

\begin{proof}[Proof of \cref{cor:bc-euler-todd-stong}]
The top Chern number genus $\chi$ and its oddification have Hirzebruch series
\[
 Q_\chi(z)=1+z,
 \qquad
 \widetilde Q_\chi(z)=\sqrt{1-z^2}.
\]
The oddification has exponential
\[
 \frac{z}{\widetilde Q_\chi(z)}
 =\frac{z}{\sqrt{1-z^2}}.
\]
This is the unique odd strict solution of
\[
 (\,f')^2=(1+f^2)^3
 =1+3f^2+3f^4+f^6,
\]
which is the specialization of $f_{\Bc,3}$ at
$(a_1,a_2,a_3)=(-3,3,-1)$.

For the Todd genus,
\[
 Q_{\Td}(z)=\frac{z}{1-e^{-z}},
 \qquad
 \widetilde Q_{\Td}(z)=\frac{z}{2\sinh(z/2)}.
\]
The exponential of the oddification is $2\sinh(z/2)$, and it satisfies
\[
 (f')^2=1+\frac14\,f^2.
\]
This is the specialization of \eqref{eq:bcN-exponential} at
$(a_1,a_2,a_3)=(-1/4,0,0)$. The result now follows from
\cref{thm:oddification-stong}.
\end{proof}

\subsection{Information not detected by the Ochanine genus and the top Chern number}
\label{sec:bc-not-determined}

\begin{proof}[Proof of \cref{thm:bc-not-determined}]
Recall the coefficient classes
\[
 \theta_{m1}=[x^m y]\Theta_1(x,y)\in\Lambda^{-4m},\qquad m\geq1.
\]
Taking the coefficient of $y$ in \eqref{eq:bc-images-theta} gives
\[
 [y]\,\Bc\bigl(\Theta_1(x,y)\bigr)
 =2+4a_1x+6a_2x^2+8a_3x^3.
\]
Thus
\[
 \Bc(\theta_{11})=4a_1,\qquad
 \Bc(\theta_{21})=6a_2,\qquad
 \Bc(\theta_{31})=8a_3,
\]
and $\Bc(\theta_{m1})=0$ for $m\geq4$. For $n\geq5$, put
\begin{equation}\label{eq:undecomposable-comparison-classes}
 \begin{aligned}
  X_n&=2\theta_{n1}+\theta_{31}\theta_{11}^{\,n-3},\\
  Y_n&=\theta_{n1}+8\theta_{31}\theta_{21}\theta_{11}^{\,n-5}.
 \end{aligned}
\end{equation}
Every summand belongs to $\Lambda^{-4n}\subset\St^{-4n}$. Multiplicativity
of the genus gives
\[
 \begin{aligned}
  \Bc(X_n)&=128(4a_1)^{n-5}a_1^2a_3,\\
  \Bc(Y_n)&=384(4a_1)^{n-5}a_2a_3.
 \end{aligned}
\]
Their difference is the nonzero polynomial stated in the theorem, and
setting $a_3=0$ gives $\Och(X_n)=\Och(Y_n)=0$. By
\cref{cor:bc-euler-todd-stong}, evaluation at $(a_1,a_2,a_3)=(-3,3,-1)$
also gives
\[
 \chi(X_n)=\chi(Y_n)=-1152(-12)^{n-5}.
\]

To prove undecomposability, let $\mathfrak a_{\Q}$ be the augmentation ideal
of $\OmU\otimes\Q$. Modulo its square, compositional inversion gives
\[
 E(x)=B^{-1}(x)
 \equiv x-\sum_{j\geq1}b_jx^{j+1}\pmod{\mathfrak a_{\Q}^{\,2}}.
\]
Put
\[
 X=B(x),\qquad Y=B(y),\qquad
 T_\pm=(\sqrt X\pm\sqrt Y)^2.
\]
Substituting the preceding congruence into the two branches
\eqref{eq:zpm-B-coordinate} gives
\[
 \Theta_1(x,y)
 \equiv T_++T_-
 -\sum_{j\geq1}b_j\bigl(T_+^{j+1}+T_-^{j+1}\bigr)
 \pmod{\mathfrak a_{\Q}^{\,2}}.
\]
Since $T_++T_-=2B(x)+2B(y)$, this term contributes no mixed monomials.
In each summand of the remaining sum there is already a factor
$b_j\in\mathfrak a_{\Q}$. Since $B(x)\equiv x$ and $B(y)\equiv y$ modulo
$\mathfrak a_{\Q}$, the binomial theorem gives
\[
 T_+^{j+1}+T_-^{j+1}
 \equiv
 2\sum_{k=0}^{j+1}\binom{2j+2}{2k}x^{j+1-k}y^k
 \pmod{\mathfrak a_{\Q}}.
\]
Multiplication by $b_j$ places the error in
$\mathfrak a_{\Q}^{\,2}$. For $m\geq1$, only $j=m$ and $k=1$ contribute
to the coefficient of $x^my$. Consequently,
\[
 \theta_{m1}
 \equiv-2\binom{2m+2}{2}b_m
 =-2(m+1)(2m+1)b_m
 \pmod{\mathfrak a_{\Q}^{\,2}}.
\]
Since $s_{2m}$ vanishes on decomposables and
$s_{2m}(b_m)=2(-1)^m$ by \cref{thm:main}, we obtain
\begin{equation}\label{eq:milnor-theta-m1}
 s_{2m}(\theta_{m1})=4(-1)^{m+1}(m+1)(2m+1)\ne0.
\end{equation}
For $n\geq5$, the product terms in
\eqref{eq:undecomposable-comparison-classes} have at least two
positive-dimensional factors. Hence
\[
 \begin{aligned}
  s_{2n}(X_n)&=2s_{2n}(\theta_{n1}),\\
  s_{2n}(Y_n)&=s_{2n}(X_n-Y_n)=s_{2n}(\theta_{n1})\ne0.
 \end{aligned}
\]
These nonzero Milnor numbers prove that $X_n$, $Y_n$, and $X_n-Y_n$ are
undecomposable, even after rationalization.
\end{proof}

\section{Proofs of the spin and string integrality theorems}
\label{sec:spin-integrality}

We first prove \cref{thm:bc-spin-below24} by comparing the Buchstaber and
Ochanine genera. In real dimensions below $24$, the possible additional
denominators reduce to halves of a few integral characteristic numbers, whose
parity is controlled by the Wu formulas. We then prove
\cref{thm:bc-string-through24}; only dimension $24$ remains, and there we
evaluate the characteristic polynomial on an integral basis of string
bordism.

\begin{proof}[Proof of \cref{thm:bc-spin-below24}]
The oriented extension in \cref{sec:spin-main} already gives
\begin{equation}\label{eq:spin-dyadic-values}
 \Bc(M)\in\mathbb Z\!\left[\frac12,a_1,a_2,a_3\right].
\end{equation}
Thus no odd prime can occur in a denominator of $\Bc(M)$, and it remains
to exclude powers of $2$. We do this by proving that the value also belongs
to $\Ztwo[a_1,a_2,a_3]$, where only odd denominators are allowed.
The classical spin integrality of the Ochanine genus gives
\[
 \Och(M)\in\mathbb Z[a_1,a_2].
\]
Indeed, for the normalization
$(f')^2=1-2\delta f^2+\varepsilon f^4$, its image on spin cobordism is
$\mathbb Z[16\delta,(8\delta)^2,\varepsilon]$
\cite[Introduction]{HH24}. Here $\delta=a_1/2$ and $\varepsilon=a_2$.
It therefore suffices to prove
$\Bc(M)-\Och(M)\in\Ztwo[a_1,a_2,a_3]$.

Write $p_i=p_i(\tau_M)$ and $w_i=w_i(\tau_M)$. The integral spin characteristic
classes $q_1\in H^4(M;\mathbb Z)$ and $q_2\in H^8(M;\mathbb Z)$ satisfy
\begin{equation}\label{eq:spin-characteristic-classes}
 p_1=2q_1,\qquad p_2=q_1^2+2q_2,\qquad
 \overline q_1=w_4,\qquad \overline q_2=w_8,
\end{equation}
where the bar denotes reduction modulo $2$
\cite[\S~5]{HHLZ22}. We also use
$\overline p_i=w_{2i}^2$ \cite[\S~15]{MS74}.

\emph{The calculation of the characteristic polynomial.}
Put $Q(z)=z/f_{\Bc}(z)$ and $Q_0(z)=Q(z)|_{a_3=0}$.
For the oriented extensions, these even series are evaluated on the square
roots of the ordinary Pontryagin roots, as in
\eqref{eq:oriented-series-Q}--\eqref{eq:oriented-extension-Q}.
The defining equation
\[
 (f_{\Bc}')^2=1-a_1f_{\Bc}^2+a_2f_{\Bc}^4-a_3f_{\Bc}^6
\]
gives
\begin{equation}\label{eq:spin-log-ratio}
 \log\frac{Q(z)}{Q_0(z)}
 =a_3\left(\frac{z^6}{14}-\frac{a_1z^8}{42}
  +\left(\frac{31a_1^2}{6930}+\frac{4a_2}{385}\right)z^{10}\right)
  +O(z^{12}),
\end{equation}
and
\[
 Q_0(z)=1+\frac{a_1}{6}z^2
  +\left(\frac{7a_1^2}{360}-\frac{a_2}{10}\right)z^4+O(z^6).
\]
These coefficients are obtained recursively from
$f_{\Bc}''=-a_1f_{\Bc}+2a_2f_{\Bc}^3-3a_3f_{\Bc}^5$ and
$f_{\Bc}(0)=0$, $f_{\Bc}'(0)=1$.

Let $u_i$ be the formal Pontryagin roots, so that
\[
 p(\tau_M)=\prod_i(1+u_i),\qquad
 p_j=\sum_{i_1<\cdots<i_j}u_{i_1}\cdots u_{i_j}.
\]
Thus the Pontryagin classes $p_j$ are the elementary symmetric functions
of the $u_i$. The power sums
\[
 s_j(p)=\sum_i u_i^j
\]
are expressed in the $p_i$ by Newton's identities. They are not additional
independent characteristic classes. In particular,
$\sum_i u_i=p_1$ and $\sum_i u_i^2=p_1^2-2p_2$.
All degree bounds below refer to ordinary cohomology on $M$:
\[
 \deg u_i=4,\qquad
 \deg p_j=\deg s_j(p)=4j.
\]
The grading $\deg a_j=-4j$ in the genus coefficient ring is separate from
this cohomological degree. The parameters $a_j$ are coefficients when we
select the component to pair with $[M]$. For example,
$a_1a_3s_4(p)$ has cohomological degree $16$ on $M$.

Put
\[
 \mathcal Q=\prod_iQ(\sqrt{u_i}),\qquad
 \mathcal Q_0=\prod_iQ_0(\sqrt{u_i}).
\]
The degree-$\dim M$ components of these classes evaluate to $\Bc(M)$ and
$\Och(M)$, respectively. Their constant terms are $1$, so the formal
logarithm of their ratio is defined. Write
\[
 D=\log\frac{\mathcal Q}{\mathcal Q_0}
   =\sum_i\log\frac{Q(\sqrt{u_i})}{Q_0(\sqrt{u_i})}.
\]
For a characteristic class $A$, let $[A]_d$ denote its component of
cohomological degree $d$, and let $[A]_{\leq d}$ denote the sum of its
components of degrees at most $d$. In \eqref{eq:spin-log-ratio}, the term
$z^{2j}$ becomes $u_i^j$ after substitution and becomes $s_j(p)$ after
summation over $i$. Hence
\[
 [D]_{\leq20}
 =a_3\left(\frac{s_3(p)}{14}-\frac{a_1s_4(p)}{42}
       +\left(\frac{31a_1^2}{6930}+\frac{4a_2}{385}\right)s_5(p)\right).
\]
The three displayed summands have cohomological degrees $12$, $16$, and
$20$. Terms coming from $O(z^{12})$ have cohomological degree at least
$24$, so they cannot contribute when $\dim M<24$.

To recover the difference of the characteristic classes from the
logarithmic difference, use the exact identity
\[
 \mathcal Q-\mathcal Q_0
 =\mathcal Q_0(e^D-1)
 =\mathcal Q_0\left(D+\frac{D^2}{2!}+\frac{D^3}{3!}+\cdots\right).
\]
The smallest possible cohomological degree in $D$ is $12$. A product of
two terms of $D$ therefore has degree at least $12+12=24$, and a product
of $r$ terms has degree at least $12r$. Consequently,
\[
 [e^D-1]_{\leq20}=[D]_{\leq20},\qquad
 [\mathcal Q-\mathcal Q_0]_{\leq20}
 =[\mathcal Q_0D]_{\leq20}.
\]
This discards powers $D^r$ with $r\geq2$, not quadratic monomials in the
parameters: the term involving $a_1a_3$, for example, is retained in $D$.

Since $D$ starts in degree $12$, only the components of $\mathcal Q_0$ of
degrees $0$, $4$, and $8$ can contribute to $\mathcal Q_0D$ through degree
$20$. From the displayed expansion of $Q_0$, we obtain
\[
 \begin{aligned}
  [\mathcal Q_0]_0&=1,\\
  [\mathcal Q_0]_4&=\frac{a_1p_1}{6},\\
  [\mathcal Q_0]_8
  &=\left(\frac{7a_1^2}{360}-\frac{a_2}{10}\right)(p_1^2-2p_2)
    +\frac{a_1^2}{36}p_2\\
  &=\frac{a_1^2(7p_1^2-4p_2)}{360}
    -\frac{a_2(p_1^2-2p_2)}{10}.
 \end{aligned}
\]
Indeed, a degree-$8$ term in the product comes either from $u_i^2$ in
one factor or from $u_i u_j$ in two distinct factors. Summing these
contributions uses $\sum_i u_i^2=p_1^2-2p_2$ and
$\sum_{i<j}u_i u_j=p_2$.
The components needed in dimensions $12$, $16$, and $20$ are therefore
\[
 \begin{aligned}
  [\mathcal Q-\mathcal Q_0]_{12}
   &=[D]_{12},\\
  [\mathcal Q-\mathcal Q_0]_{16}
   &=[D]_{16}+[\mathcal Q_0]_4[D]_{12},\\
  [\mathcal Q-\mathcal Q_0]_{20}
   &=[D]_{20}+[\mathcal Q_0]_4[D]_{16}
     +[\mathcal Q_0]_8[D]_{12}.
 \end{aligned}
\]
These identities list all ways in which the cohomological degrees of the
two factors can add to the dimension of $M$.

We next rewrite these components using integral spin characteristic
classes. The Newton identities needed here are
\[
 \begin{aligned}
  s_3(p)&=p_1^3-3p_1p_2+3p_3,\\
  s_4(p)&=p_1^4-4p_1^2p_2+4p_1p_3+2p_2^2-4p_4,\\
  s_5(p)&=p_1^5-5p_1^3p_2+5p_1^2p_3+5p_1p_2^2
          -5p_1p_4-5p_2p_3+5p_5.
 \end{aligned}
\]
Substituting these Newton identities into the three components above and
then using $p_1=2q_1$ and $p_2=q_1^2+2q_2$ from
\eqref{eq:spin-characteristic-classes} gives the exact expressions
\[
 \begin{aligned}
  [\mathcal Q-\mathcal Q_0]_{12}
   &=\frac{a_3}{14}\bigl(2q_1^3-12q_1q_2+3p_3\bigr),\\[4pt]
  [\mathcal Q-\mathcal Q_0]_{16}
   &=\frac{a_1a_3}{42}
     \bigl(12q_1^2q_2-8q_2^2-5q_1p_3+4p_4\bigr),
 \end{aligned}
\]
and
\[
 \begin{aligned}
  [\mathcal Q-\mathcal Q_0]_{20}
   &=\frac{a_1^2a_3}{6930}
     \bigl(18q_1^5-338q_1^3q_2+932q_1q_2^2\\
   &\hspace{3em}{}+124q_1^2p_3-343q_2p_3-90q_1p_4+155p_5\bigr)\\[4pt]
   &\quad+\frac{a_2a_3}{770}
     \bigl(-6q_1^5-144q_1^3q_2+56q_1q_2^2\\
   &\hspace{3em}{}+87q_1^2p_3-14q_2p_3-80q_1p_4+40p_5\bigr).
 \end{aligned}
\]
Every monomial in $q_1,q_2,p_3,p_4,p_5$ appearing here is an integral
cohomology class. Thus, for integrality
at $2$, a term whose rational coefficient has odd denominator needs no
further argument. We collect coefficients after the spin substitutions
and separate these terms from the possible half-integral terms.
For example, in cohomological degree $12$,
\[
 \frac{s_3(p)}{14}
 =\frac{2q_1^3-12q_1q_2+3p_3}{14}
 =\frac{p_3}{2}
  +\frac{q_1^3-6q_1q_2-2p_3}{7}.
\]
The second summand has odd denominator, so only the parity of
$\langle p_3,[M]\rangle$ can obstruct integrality at $2$ in this dimension.
In degree~$16$ the component is a multiple of $a_1a_3$. In degree~$20$,
we apply the same procedure separately to the coefficients of $a_1^2a_3$
and $a_2a_3$. Evaluating on $[M]$ gives
\begin{equation}\label{eq:spin-half-obstructions}
\begin{aligned}
 \dim M=12:\quad
 \Bc(M)-\Och(M)
 &\equiv\frac{a_3}{2}\langle p_3,[M]\rangle,\\
 \dim M=16:\quad
 \Bc(M)-\Och(M)
 &\equiv\frac{a_1a_3}{2}\langle q_1p_3,[M]\rangle,\\
 \dim M=20:\quad
 \Bc(M)-\Och(M)
 &\equiv\frac{a_1^2a_3}{2}\langle q_2p_3+p_5,[M]\rangle\\
 &\quad+\frac{a_2a_3}{2}\langle q_1^2p_3,[M]\rangle.
\end{aligned}
\end{equation}
Each congruence is taken modulo $\Ztwo[a_1,a_2,a_3]$. Equivalently, after substituting
$p_1=2q_1$ and $p_2=q_1^2+2q_2$, all omitted coefficients have odd
denominators. In dimensions $0$, $4$, and $8$, the two genera coincide,
because the correction starts in cohomological degree $12$.
Thus \eqref{eq:spin-half-obstructions} reduces the proof to the evenness of
four integral characteristic numbers. It does not yet establish their
parity.

\emph{The parity calculation.}
An integral characteristic number is even precisely when the reduction
of its characteristic class evaluates to zero on $[M]_2$. Using the
reductions in \eqref{eq:spin-characteristic-classes} and
$\overline p_i=w_{2i}^2$, the characteristic numbers in
\eqref{eq:spin-half-obstructions} reduce to the evaluations of
\[
 w_6^2,\qquad w_4w_6^2,\qquad
 w_8w_6^2+w_{10}^2,\qquad w_4^2w_6^2
\]
on the mod-$2$ fundamental class in dimensions $12$, $16$, $20$, and $20$,
respectively. We show that these evaluations vanish.
For a spin manifold, $w_1=w_2=w_3=0$, and its Wu classes satisfy
$v_2=0$, $v_4=w_4$. Thus, for classes of the appropriate degrees,
\begin{equation}\label{eq:spin-wu-pairings}
 \langle\operatorname{Sq}^2z,[M]_2\rangle=0,\qquad
 \langle\operatorname{Sq}^4z,[M]_2\rangle
 =\langle w_4z,[M]_2\rangle
\end{equation}
\cite[Theorem~11.14]{MS74}.
Since $w_4$ and $w_8$ lift integrally by
\eqref{eq:spin-characteristic-classes}, their first Steenrod squares vanish.
Wu's formula for the squares of Stiefel--Whitney classes
\cite[Problem~8-A]{MS74} gives
\[
 \operatorname{Sq}^2w_4=w_6,\quad
 \operatorname{Sq}^2w_6=0,\quad
 \operatorname{Sq}^2w_8=w_{10},\quad
 \operatorname{Sq}^2w_{10}=0.
\]
The Cartan formula now yields
\[
\begin{aligned}
 \operatorname{Sq}^2(w_4w_6)&=w_6^2,&
 \operatorname{Sq}^4(w_6^2)&=0,\\
 \operatorname{Sq}^2(w_4^3w_6)&=w_4^2w_6^2,&
 \operatorname{Sq}^2(w_8w_{10})&=w_{10}^2.
\end{aligned}
\]
Evaluation using \eqref{eq:spin-wu-pairings} proves the required vanishing
for $w_6^2$, $w_4w_6^2$, and $w_4^2w_6^2$, as well as for $w_{10}^2$.
Finally,
\[
\begin{aligned}
 \operatorname{Sq}^2(w_4w_6w_8)&=w_6^2w_8+w_4w_6w_{10},\\
 \operatorname{Sq}^4(w_6w_{10})&=w_{10}^2.
\end{aligned}
\]
For the second identity, use
$\operatorname{Sq}^4w_6=w_4w_6+w_{10}$,
$\operatorname{Sq}^4w_{10}=w_4w_{10}$, and
$\operatorname{Sq}^3w_6=\operatorname{Sq}^3w_{10}=0$. The latter follow
from $\operatorname{Sq}^3=\operatorname{Sq}^1\operatorname{Sq}^2$.
Consequently, in dimension $20$,
\[
 \langle w_6^2w_8,[M]_2\rangle
 =\langle w_4w_6w_{10},[M]_2\rangle
 =\langle w_{10}^2,[M]_2\rangle=0.
\]
All the integral characteristic numbers occurring with a factor $1/2$ in
\eqref{eq:spin-half-obstructions} are therefore even. This proves
$\Bc(M)\in\Ztwo[a_1,a_2,a_3]$ in dimensions $12$, $16$, and $20$.
In dimensions $0$, $4$, and $8$, the correction
$\mathcal Q-\mathcal Q_0$ has no component of top degree, since it starts in
cohomological degree $12$. Hence
\[
 \Bc(M)=\Och(M)\in\mathbb Z[a_1,a_2].
\]
If $\dim M$ is less than $24$ and is not divisible by $4$, then $\mathcal Q$
has no component in degree $\dim M$: it is expressed in Pontryagin classes, whose
cohomological degrees are multiples of $4$. Thus $\Bc(M)=0$.
These cases, together with dimensions $12$, $16$, and $20$, exhaust the
range $\dim M<24$.
Together with \eqref{eq:spin-dyadic-values} and
\[
 \mathbb Z\!\left[\frac12,a_1,a_2,a_3\right]
 \cap\Ztwo[a_1,a_2,a_3]=\mathbb Z[a_1,a_2,a_3],
\]
this proves the theorem. For a Calabi--Yau manifold, $c_1(\tau_M)=0$ in $H^2(M;\mathbb Z)$, so
$w_2(\tau_M)=c_1(\tau_M)\bmod2=0$.
\end{proof}

\begin{proof}[Proof of \cref{thm:bc-string-through24}]
Every string manifold is spin, so \cref{thm:bc-spin-below24} treats all
dimensions below $24$. It remains to consider a closed string manifold $M$ of
real dimension $24$. In particular, $p_1(\tau_M)=0$.

For $j\geq1$, put
\[
 \ell_j=[z^{2j}]\log Q(z).
\]
The coefficients needed in dimension $24$, obtained recursively from the
defining differential equation for $f_{\Bc}$, are
\begin{equation}\label{eq:string-log-coefficients}
\begin{aligned}
 \ell_2&=\frac{a_1^2}{180}-\frac{a_2}{10},\\
 \ell_3&=\frac{a_1^3}{2835}+\frac{a_1a_2}{105}
          +\frac{a_3}{14},\\
 \ell_4&=\frac{a_1^4}{37800}-\frac{a_1^2a_2}{1050}
          -\frac{a_1a_3}{42}-\frac{a_2^2}{300},\\
 \ell_6&=\frac{691a_1^6}{3831077250}
          -\frac{191a_1^4a_2}{70945875}
          -\frac{514a_1^3a_3}{945945}
          -\frac{5933a_1^2a_2^2}{47297250}\\
 &\hspace{2em}{}
          -\frac{562a_1a_2a_3}{105105}
          -\frac{a_2^3}{4875}
          -\frac{11a_3^2}{2548}.
\end{aligned}
\end{equation}
Since $p_1=0$, one has
\[
 \log\mathcal Q=\sum_{j\geq2}\ell_js_j(p).
\]
Only the partitions $6$, $4+2$, $3+3$, and $2+2+2$ can contribute to its
exponential in cohomological degree $24$. Newton's identities give
\[
 s_2(p)=-2p_2,\qquad s_3(p)=3p_3,\qquad
 s_4(p)=2p_2^2-4p_4,
\]
and
\[
 s_6(p)=-2p_2^3+3p_3^2+6p_2p_4-6p_6.
\]
Consequently,
\begin{equation}\label{eq:string-degree24-class}
\begin{aligned}
 [\mathcal Q]_{24}
 &=\left(-2\ell_6-4\ell_2\ell_4-\frac43\ell_2^3\right)p_2^3\\
 &\quad+\left(3\ell_6+\frac92\ell_3^2\right)p_3^2
       +\left(6\ell_6+8\ell_2\ell_4\right)p_2p_4
       -6\ell_6p_6.
\end{aligned}
\end{equation}

Han and Huang give an integral basis $[M_1],\ldots,[M_4]$ of
$\Omega_{24}^{\mathrm{String}}$
\cite[Theorem~1 and Corollary~3]{HH22String} and compute its Pontryagin
numbers \cite[Theorem~2.1]{HH22String}. Write a homogeneous polynomial of
weight $6$ as its coefficient vector in the ordered monomial basis
\[
 \mathcal A=
 (a_1^6,a_1^4a_2,a_1^3a_3,a_1^2a_2^2,
   a_1a_2a_3,a_2^3,a_3^2).
\]
Substitution of those characteristic numbers into
\eqref{eq:string-degree24-class} gives
\begin{equation}\label{eq:string-basis-values}
\begin{aligned}
 [\Bc(M_1)]_{\mathcal A}
  &=(0,6144,23040,-13824,622080,331776,2332800),\\
 [\Bc(M_2)]_{\mathcal A}
  &=(4096,-30720,714240,3552768,19284480,-13602816,72316800),\\
 [\Bc(M_3)]_{\mathcal A}
  &=(0,0,-32,-16,576,288,0),\\
 [\Bc(M_4)]_{\mathcal A}
  &=(0,0,0,0,28,8,12).
\end{aligned}
\end{equation}
Every entry is an integer. Since the four classes form an integral basis,
additivity of the genus gives
\[
 \Bc(M)\in\mathbb Z[a_1,a_2,a_3]
\]
for every closed string $24$-manifold $M$.
\end{proof}

\begin{example}[(The Anderson--Brown--Peterson manifold)]
\label[example]{ex:abp-spin24}
By Anderson--Brown--Peterson \cite{ABP67}, there exists a closed smooth
spin $24$-manifold $\Sigma_{\mathrm{ABP}}$ for which
\[
 \bigl\langle s_6(p),[\Sigma_{\mathrm{ABP}}]\bigr\rangle
 \equiv1\pmod2.
\]
See also \cite[\S~1.3]{AS25} for this precise formulation.
As in the preceding proof, $p_i=p_i(\tau_M)$ and
$s_j(p)=\sum_i u_i^j$ for the ordinary Pontryagin roots $u_i$.
We now show directly that the oddness of this characteristic number forces
$\Bc(\Sigma_{\mathrm{ABP}})$ to be nonintegral.

It suffices to examine the coefficient of $a_3^2$. Set $a_1=a_2=0$, and
write $f=\left.f_{\Bc}\right|_{a_1=a_2=0}$ and $Q_*(z)=z/f(z)$.
The equation $(f')^2=1-a_3f^6$, with $f(0)=0$ and $f'(0)=1$, gives
\[
 \begin{aligned}
  f(z)&=z-\frac{a_3}{14}z^7+\frac{5a_3^2}{728}z^{13}+O(z^{19}),\\
  Q_*(z)&=1+\frac{a_3}{14}z^6
            -\frac{9a_3^2}{5096}z^{12}+O(z^{18}).
 \end{aligned}
\]
After substituting $z=\sqrt{u_i}$, a term of cohomological degree $24$
in $\prod_iQ_*(\sqrt{u_i})$ comes either from $u_i^6$ in one factor or
from $u_i^3u_j^3$ in two distinct factors. Thus the coefficient of $a_3^2$
in that characteristic class is
\[
 -\frac{9}{5096}\sum_i u_i^6
 +\frac{1}{196}\sum_{i<j}u_i^3u_j^3.
\]
The second summand is a product of two degree-$12$ contributions. It must
be retained in dimension $24$, unlike in the preceding proof.
Since
\[
 s_3(p)^2=s_6(p)+2\sum_{i<j}u_i^3u_j^3,
\]
we obtain, for every closed oriented $24$-manifold $M$,
\begin{equation}\label{eq:abp-a3-square}
 [a_3^2]\Bc(M)
 =\frac{13\bigl\langle s_3(p)^2,[M]\bigr\rangle
             -22\bigl\langle s_6(p),[M]\bigr\rangle}{5096}.
\end{equation}
Both classes in the numerator have cohomological degree $24$.
The symmetric polynomial $\sum_{i<j}u_i^3u_j^3$ is an integral polynomial
in the Pontryagin classes. Consequently,
\[
 s_3(p)^2\equiv s_6(p)\pmod2.
\]
For $M=\Sigma_{\mathrm{ABP}}$, both characteristic numbers in
\eqref{eq:abp-a3-square} are therefore odd. Its numerator is odd, whereas
$5096=2^3\cdot7^2\cdot13$. Hence
\[
 \nu_2\!\left([a_3^2]\Bc(\Sigma_{\mathrm{ABP}})\right)=-3.
\]
The oriented extension takes values in
$\mathbb Z[1/2,a_1,a_2,a_3]$, so no odd prime remains in the reduced
denominator. This coefficient therefore has exact denominator $8$, and
\[
 \Bc(\Sigma_{\mathrm{ABP}})\notin\mathbb Z[a_1,a_2,a_3].
\]
This proves that $24$ is the first real dimension in which integrality can
fail on spin manifolds.

The class $[\Sigma_{\mathrm{ABP}}]$ is also undecomposable in spin
cobordism, even after rationalization: the number
$\langle s_6(p),[M]\rangle$ vanishes on every product of positive-dimensional
manifolds of total dimension $24$, since the power sum is additive under
Whitney sums and each factor has dimension less than $24$.
By contrast, every decomposable spin class in dimension $24$ has integral
Buchstaber genus by \cref{thm:bc-spin-below24} and multiplicativity.
\end{example}

\printbibliography[heading=bibliography]

\bigskip

\noindent\begin{minipage}{\linewidth}
\noindent
{\itshape Victor Buchstaber}\\
Steklov Mathematical Institute of Russian Academy of Sciences\\
Email: \texttt{buchstab@mi-ras.ru}

\vspace{1.5em}

\noindent
{\itshape Mikhail Kornev}\\
HSE University\\
Email: \texttt{mkorneff@mi-ras.ru}

\end{minipage}


\end{document}